\documentclass{siamart190516}
\usepackage[utf8]{inputenc}
\usepackage{amsfonts}
\usepackage{amssymb}

\usepackage[pagewise,mathlines]{lineno}
\usepackage{amsmath} 
\usepackage{etoolbox}          

\newcommand*\linenomathpatch[1]{%
  \cspreto{#1}{\linenomath}%
  \cspreto{#1*}{\linenomath}%
  \csappto{end#1}{\endlinenomath}%
  \csappto{end#1*}{\endlinenomath}%
}
\newcommand*\linenomathpatchAMS[1]{%
  \cspreto{#1}{\linenomathAMS}%
  \cspreto{#1*}{\linenomathAMS}%
  \csappto{end#1}{\endlinenomath}%
  \csappto{end#1*}{\endlinenomath}%
}

\expandafter\ifx\linenomath\linenomathWithnumbers
  \let\linenomathAMS\linenomathWithnumbers
  \patchcmd\linenomathAMS{\advance\postdisplaypenalty\linenopenalty}{}{}{}
\else
  \let\linenomathAMS\linenomathNonumbers
\fi

\linenomathpatch{equation}
\linenomathpatchAMS{gather}
\linenomathpatchAMS{multline}
\linenomathpatchAMS{align}
\linenomathpatchAMS{alignat}
\linenomathpatchAMS{flalign}

\makeatletter
\patchcmd{\mmeasure@}{\measuring@true}{
  \measuring@true
  \ifnum-\linenopenaltypar>\interdisplaylinepenalty
    \advance\interdisplaylinepenalty-\linenopenalty
  \fi
  }{}{}
\makeatother

\usepackage{siunitx}

\usepackage{graphicx}

\usepackage{xcolor}
\usepackage{tikz}
\tikzset{every picture/.style={line width=0.11mm}}
\newcommand{\operpsymbol}{\begin{tikzpicture}[scale=0.112]
  \draw (0,-0.5)--(0,1); \draw (-0.866,-0.5)--(0.866,-0.5);
  \draw (0,0) circle [radius=1];
\end{tikzpicture}}
\newcommand{\operp}{\mathbin{\raisebox{-1pt}{\operpsymbol}}}

\newcommand{\vertiii}[1]{{\left\vert\kern-0.25ex\left\vert\kern-0.25ex\left\vert #1 
    \right\vert\kern-0.25ex\right\vert\kern-0.25ex\right\vert}}

\makeatletter
\let\save@mathaccent\mathaccent
\newcommand*\if@single[3]{%
  \setbox0\hbox{${\mathaccent"0362{#1}}^H$}%
  \setbox2\hbox{${\mathaccent"0362{\kern0pt#1}}^H$}%
  \ifdim\ht0=\ht2 #3\else #2\fi
  }
\newcommand*\rel@kern[1]{\kern#1\dimexpr\macc@kerna}
\newcommand*\widebar[1]{\@ifnextchar^{{\wide@bar{#1}{0}}}{\wide@bar{#1}{1}}}
\newcommand*\wide@bar[2]{\if@single{#1}{\wide@bar@{#1}{#2}{1}}{\wide@bar@{#1}{#2}{2}}}
\newcommand*\wide@bar@[3]{%
  \begingroup
  \def\mathaccent##1##2{%
    \let\mathaccent\save@mathaccent
    \if#32 \let\macc@nucleus\first@char \fi
    \setbox\z@\hbox{$\macc@style{\macc@nucleus}_{}$}%
    \setbox\tw@\hbox{$\macc@style{\macc@nucleus}{}_{}$}%
    \dimen@\wd\tw@
    \advance\dimen@-\wd\z@
    \divide\dimen@ 3
    \@tempdima\wd\tw@
    \advance\@tempdima-\scriptspace
    \divide\@tempdima 10
    \advance\dimen@-\@tempdima
    \ifdim\dimen@>\z@ \dimen@0pt\fi
    \rel@kern{0.6}\kern-\dimen@
    \if#31
      \overline{\rel@kern{-0.6}\kern\dimen@\macc@nucleus\rel@kern{0.4}\kern\dimen@}%
      \advance\dimen@0.4\dimexpr\macc@kerna
      \let\final@kern#2%
      \ifdim\dimen@<\z@ \let\final@kern1\fi
      \if\final@kern1 \kern-\dimen@\fi
    \else
      \overline{\rel@kern{-0.6}\kern\dimen@#1}%
    \fi
  }%
  \macc@depth\@ne
  \let\math@bgroup\@empty \let\math@egroup\macc@set@skewchar
  \mathsurround\z@ \frozen@everymath{\mathgroup\macc@group\relax}%
  \macc@set@skewchar\relax
  \let\mathaccentV\macc@nested@a
  \if#31
    \macc@nested@a\relax111{#1}%
  \else
    \def\gobble@till@marker##1\endmarker{}%
    \futurelet\first@char\gobble@till@marker#1\endmarker
    \ifcat\noexpand\first@char A\else
      \def\first@char{}%
    \fi
    \macc@nested@a\relax111{\first@char}%
  \fi
  \endgroup
}
\makeatother

\usetikzlibrary{cd}

\usepackage{pgfplotstable}
\newcommand{\logLogSlopeTriangle}[5]
{

    \pgfplotsextra
    {
        \pgfkeysgetvalue{/pgfplots/xmin}{\xmin}
        \pgfkeysgetvalue{/pgfplots/xmax}{\xmax}
        \pgfkeysgetvalue{/pgfplots/ymin}{\ymin}
        \pgfkeysgetvalue{/pgfplots/ymax}{\ymax}

        \pgfmathsetmacro{\xArel}{#1}
        \pgfmathsetmacro{\yArel}{#3}
        \pgfmathsetmacro{\xBrel}{#1-#2}
        \pgfmathsetmacro{\yBrel}{\yArel}
        \pgfmathsetmacro{\xCrel}{\xArel}

        \pgfmathsetmacro{\lnxB}{\xmin*(1-(#1-#2))+\xmax*(#1-#2)} 
        \pgfmathsetmacro{\lnxA}{\xmin*(1-#1)+\xmax*#1} 
        \pgfmathsetmacro{\lnyA}{\ymin*(1-#3)+\ymax*#3} 
        \pgfmathsetmacro{\lnyC}{\lnyA+#4*(\lnxA-\lnxB)}
        \pgfmathsetmacro{\yCrel}{\lnyC-\ymin)/(\ymax-\ymin)} 

        \coordinate (A) at (rel axis cs:\xArel,\yArel);
        \coordinate (B) at (rel axis cs:\xBrel,\yBrel);
        \coordinate (C) at (rel axis cs:\xCrel,\yCrel);

        \draw[black]   (A)-- node[pos=0.5,anchor=north] {\scriptsize{1}}
                    (B)-- 
                    (C)-- node[pos=0.,anchor=west] {\scriptsize{\color{#5}#4}} 
                    (A);
    }
}

\newsiamremark{remark}{Remark}

\numberwithin{equation}{section}

\usepackage[
backend=biber,
style=numeric,
sorting=nty,
url=false,
isbn=false,
doi=false
]{biblatex}
\usepackage{microtype}

\DeclareFieldFormat*{title}{\mkbibemph{#1}}
\DeclareFieldFormat*{citetitle}{\mkbibemph{#1}}
\DeclareFieldFormat{journaltitle}{#1}

\renewbibmacro*{in:}{%
  \ifentrytype{article}
    {}
    {\printtext{\bibstring{in}\intitlepunct}}}

\newbibmacro*{pubinstorg+location+date}[1]{%
  \printlist{#1}%
  \newunit
  \printlist{location}%
  \newunit
  \usebibmacro{date}%
  \newunit}

\renewbibmacro*{publisher+location+date}{\usebibmacro{pubinstorg+location+date}{publisher}}
\renewbibmacro*{institution+location+date}{\usebibmacro{pubinstorg+location+date}{institution}}
\renewbibmacro*{organization+location+date}{\usebibmacro{pubinstorg+location+date}{organization}}

\usepackage{hyperref}
\makeatletter
\renewcommand*{\eqref}[1]{%
  \hyperref[{#1}]{\textup{\tagform@{\ref*{#1}}}}%
}
\makeatother

\title{An arbitrary order and pointwise divergence-free finite element scheme for the incompressible 3D Navier-Stokes equations.}
\author{Marien-Lorenzo HANOT.\thanks{ IMAG, UMR 5149, 
 Universit\'e de Montpellier, 
34090 Montpellier, FRANCE
(\email{marien-lorenzo.hanot@umontpellier.fr}).
}}

\begin{document}
\maketitle
\begin{abstract}
In this paper we discretize the incompressible Navier-Stokes equations in the framework of finite element exterior calculus.
We make use of the Lamb identity to rewrite the equations into a vorticity-velocity-pressure form which fits into the de Rham complex of minimal regularity.
We propose a discretization on a large class of finite elements, 
including arbitrary order polynomial spaces readily available in many libraries.
The main advantage of this discretization is that the divergence of the fluid velocity is pointwise zero at the discrete level.
This exactness ensures pressure robustness. 
We focus the analysis on a class of linearized equations for which we prove well-posedness and provide a priori error estimates.
The results are validated with numerical simulations.
\end{abstract}
\begin{keywords} Finite Element, exterior calculus, incompressible Navier-Stokes, Hodge decomposition, de Rham complex, mixed element.
\end{keywords}

\begin{AMS} 35Q30, 65N30, 76D07, 76M10.
\end{AMS}

\section{Introduction} 
We are interested in numerical schemes preserving the algebraic structure of the incompressible Navier-Stokes equations.
Recently much work has been done to design structure preserving methods, but
while the construction of such methods was found early on in two dimensions,
the three-dimensional case remained difficult and the introduction of the finite element exterior calculus brought a significant breakthrough.
An excellent review is given by V. John et al.\ \cite{John2017}.
The general idea taken from the finite element exterior calculus is to use a subcomplex of the De Rham complex.
There are well-known discrete counterparts of this complex with minimal regularity, however the discretization of smoother variants is still an active topic 
usually leading to shape functions of high degree, see e.g.\ \cite{Neilan2015}.
We chose to use the complex with minimal regularity, as this is often done for electromagnetism or recently for magnetohydrodynamics (see \cite{Hu2021}).

The main difference from usual schemes lies in the regularity of the velocity field
since we only require it to be in $H(\mathrm{div})$ and in the discrete adjoint of $H(\mathrm{curl})$.
Although the continuous space regularity is the same as the usual one, since the adjoint of $(\mathrm{curl},H(\mathrm{curl}))$ is $(\mathrm{curl},H_0(\mathrm{curl}))$, 
and the velocity is sought in $H(\mathrm{div}) \cap H_0(\mathrm{curl}) \subset H^1$ (for a smooth enough domain, see \cite[Part 3.2]{GiraultRaviart}).
This does not hold (in general) in the discrete case since for $V_h \subset H(\mathrm{curl})$ and $(\mathrm{curl}^*,V_h^*)$
the adjoint of $(\mathrm{curl},V_h)$ we no longer have $V_h^* \subset H(\mathrm{curl})$.
This has a fundamental impact both from the philosophical and practical point of view.
In practice $v \in H(\mathrm{div})$ (resp. $v \in H(\mathrm{curl})$) does not impose continuity of the tangential (resp. normal) components on faces.
This suggests that 
we will not have any degree of freedom corresponding to these
and lack any way to set them in a Dirichlet boundary condition.
This means that the normal and the tangential part of the boundary condition must be treated in two different ways.
It makes more sense in the exterior algebra and means that the fluid velocity is really sought as a $2$-form 
(mostly defined by its flux across cell boundaries) which happens to be regular enough to also be in the domain of the exterior derivative adjoint.

Let us summarize the main idea of the algorithm. In order to preserve the divergence free constraint, we have to consider $u$ as a $2$-form,
which can be discretized by face elements. 
Then it is not straightforward to discretize the Laplacian in the usual way $\langle \nabla u, \nabla v \rangle$ because $\nabla u$ is not a natural quantity for a $2$-form.
Our simple trick is to use $\nabla \cdot u = 0$ to rewrite the Laplacian: 
\begin{equation} \label{eq:laplace}
\begin{aligned}
\Delta u =&\, \nabla (\nabla \cdot u) - \nabla \times (\nabla \times u) 
=\, - \nabla \times (\nabla \times u).
\end{aligned}
\end{equation}

Let $\Omega$ be a bounded domain of $\mathbb{R}^3$ and $T > 0$, we recall the Navier-Stokes equations:
\begin{equation}
\begin{aligned}
u_t + (u \cdot \nabla) u - \nu \Delta u + \nabla p =&\ f \text{ on } \Omega \times (0,T),\\
\nabla \cdot u =&\ 0 \text{ on } \Omega \times (0,T)
\end{aligned}
\end{equation}
together with some boundary and initial conditions,
where $u$ is the velocity of the fluid,
$p$ the pressure, $\nu$ the kinematic viscosity and $f$ an external force.\\
Using the Lamb identity $(u \cdot \nabla) u = (\nabla \times u) \times u + \frac{1}{2} \nabla (u \cdot u)$ and Equation \eqref{eq:laplace}, we get the following formulation:
\begin{equation}
\begin{aligned}
u_t + (\nabla \times u) \times u + \nu \nabla \times (\nabla \times u) + \nabla P =&\ f \text{ on } \Omega \times (0,T),\\
\nabla \cdot u =&\ 0 \text{ on } \Omega \times (0,T)
\end{aligned}
\end{equation}
where $P := p + \frac{1}{2} u \cdot u$ is the Bernoulli pressure.

Since $u$ is a $2$-form it is not natural to take $\nabla \times u$ (and it is unadvisable for reasons detailed in Remark \ref{rem:removedstar}).
Therefore, we introduce an auxiliary variable $\omega = \nabla \times u$ (namely the vorticity) and work with a mixed problem.
This is known as the vorticity-velocity-pressure formulation and was considered by many others (see \cite{ARNOLD2012,Amoura,Kreeft,Dubois,Anaya}).
The finite element exterior calculus framework is very flexible because it
allows us to work in abstract  spaces which can be discretized easily provided that some exactness properties are fulfilled.
Therefore, we shall use a generic name for our spaces, here $V^1 \times V^2 \times V^3$ for the continuous spaces and $V_h^1 \times V^2_h \times V^3_h$ for the discrete spaces (indexed by the mesh size $h$).
Stating the exact requirements for these spaces requires introducing some concepts and notation, 
hence for the sake of readability we postpone the definition to Section \ref{Spaces}.
Typically a valid choice is to take $V^1 = H(\text{curl},\Omega)$, $V^2 = H(\text{div},\Omega)$ and $V^3 = L^2(\Omega)$.
We also need another space $\mathfrak{H}^3 \subset V^3$. This is a vector space that does not only depend on $V^3$ but on the couple $V^2 \times V^3$,
and which is typically of small dimension (i.e.\ $0$ or $1$).
\begin{remark}
The choice of boundary conditions is encoded in the choice of $V^1 \times V^2 \times V^3$.
More details are given later in Section \ref{Boundaryconditions}.
\end{remark}

An example of discrete in time, mixed and linearized weak formulation is:\\
Given $f^n \in L^2(\Omega)$, find $(\omega^n,u^n,p^n,\phi^n) \in V^1 \times V^2 \times V^3 \times \mathfrak{H}^3$ such that 
$\forall (\tau,v,q,\chi) \in V^1 \times V^2 \times V^3 \times \mathfrak{H}^3$,

\begin{subequations} \label{eq:NS_Eulercont}
\begin{align}
\langle \omega^n,\tau \rangle - \langle u^n,\nabla \times \tau \rangle =&\ 0, \label{eq:weakl1}\\
\langle \frac{1}{\delta t} u^n, v \rangle + \langle \nu \nabla \times \omega^n + \theta \omega^n \times u^{n-1} + (1 - \theta) \omega^{n-1} \times u^n,v \rangle& \nonumber \\
- \langle p^n, \nabla \cdot v \rangle
=&\ \langle \frac{1}{\delta t} u^{n-1}, v \rangle + \langle f^n,v \rangle, \label{eq:weakl2}\\
\langle \nabla \cdot u^n + \phi^n,q \rangle =&\ 0, \label{eq:weakl3}\\
\langle p^n,\chi \rangle =&\ 0. \label{eq:weakl4}
\end{align}
\end{subequations}
Here $n \geq 1$ is the index of a linearly implicit time discretization 
and $\theta \in [0,1]$ is an arbitrary parameter.
In the following we take $\theta = \frac{1}{2}$. 
Let us look at a specific example:
If we take $V^1 = H(\text{curl},\Omega)$, $V^2 = H(\text{div},\Omega)$ and $V^3 = L^2(\Omega)$ 
then we must have $\mathfrak{H}^3 = \lbrace 0 \rbrace$.
Equation \eqref{eq:weakl1} is equivalent to $\omega^n = \nabla \times u^n$ and $u^n \in H_0(\mathrm{curl}, \Omega)$,
\eqref{eq:weakl2} is equivalent to 
$\frac{u^n - u^{n-1}}{\delta t} + \nu \nabla \times (\nabla \times u^n) + \frac{1}{2} (\omega^n \times u^{n-1} + \omega^{n-1} \times u^n) + \nabla p^n = f^n$ and $p^n \in H_0^1$,
\eqref{eq:weakl3} is equivalent to $\nabla \cdot u^n = 0$,
and \eqref{eq:weakl4} is trivial here (since $\mathfrak{H}^3 = \lbrace 0 \rbrace$).
This formulation is similar to the one studied by Anaya et al.\ \cite{Anaya}. 
The main difference is that our formulation is studied in the framework of finite element exterior calculus and for arbitrary low order perturbations (see Equation \eqref{eq:introformulation} below). 
In particular the abstraction made on discrete spaces allows using any discrete subcomplex (as defined in Section \ref{Notations}). 
Two families are given as examples in Section \ref{Spaces} but more exist on different kind of meshes (e.g.\ the cubical elements \cite[Chapter~7.7]{feec-cbms}).
The construction of such families is still an active topic 
and the independence over the choice of discrete subcomplex is a great feature of finite element exterior calculus 
allowing to choose any family without modification to the proofs.

\begin{remark}[On Equations \eqref{eq:weakl3} and \eqref{eq:weakl4}]
From the exterior calculus point of view $\phi$ and $\chi$ are harmonic forms of $V^3$, 
in practice they can be viewed as Lagrange multipliers.
For the given example $\mathfrak{H}^3$ is trivial but if instead we consider 
$V^1 = H_0(\text{curl},\Omega)$, $V^2 = H_0(\text{div},\Omega)$ and $V^3 = L^2(\Omega)$ 
then we must have $\mathfrak{H}^3 \approx \mathbb{R}$ (the space of constant functions).
Equation \eqref{eq:weakl3} ensures that the system is onto, although here the right-hand side is null so $\phi^n$ will always be zero.
Equation \eqref{eq:weakl4} ensures that $p^n$ is orthogonal to the space of harmonic $3$-forms (i.e.\ here that $\int_\Omega p^n = 0$).
\end{remark}

Abstracting the linearization and time discretization scheme we simply consider two linear maps: 
$l_3$ and $l_5$ defined on $L^2(\Omega) \rightarrow L^2(\Omega)$ 
(the name of which follows the convention of Arnold \& Li \cite{Arnold2016}).
And we define the problem:
Given $f_2, f_3 \in L^2(\Omega)$, find $(\omega,u,p,\phi) \in V^1 \times V^2 \times V^3 \times \mathfrak{H}^3$ such that 
$\forall (\tau,v,q,\chi) \in V^1 \times V^2 \times V^3 \times \mathfrak{H}^3$,
\begin{equation} \label{eq:introformulation} 
\begin{aligned}
\langle \omega,\tau \rangle - \langle u,\nabla \times \tau \rangle =&\ 0, \\
\langle \nu \nabla \times \omega + l_3 \omega + l_5 u,v \rangle - \langle p, \nabla \cdot v \rangle
=&\ \langle f_2,v \rangle, \\
\langle \nabla \cdot u,q \rangle + \langle \phi,q \rangle =&\ \langle f_3, q \rangle,\\
\langle p,\chi \rangle =&\ 0.
\end{aligned}
\end{equation}
We easily see that a suitable choice of $l_3$ and $l_5$ (namely $l_3 = (v \rightarrow \frac{1}{2} v \times u^{n-1})$ and $l_5 = (v \rightarrow \frac{1}{\delta t} v + \frac{1}{2} \omega^{n-1} \times v$)
allows us to recover \eqref{eq:NS_Eulercont} for $\theta = \frac{1}{2}$.
We redefine the problem in the framework of exterior calculus in \eqref{eq:defmixedcontinuous}. 
Under mild assumptions on $l_3$, $l_5$ and $\Omega$ detailed in \eqref{eq:regularity35} we prove the well-posedness of the problem \eqref{eq:defmixedcontinuous} (or equivalently \eqref{eq:introformulation}) and of its discrete counterpart \eqref{eq:defmixeddiscrete}.
If we write $(\omega,u,p,\phi)$ (resp. $(\omega_h,u_h,p_h,\phi_h)$) the solution of \eqref{eq:defmixedcontinuous} (resp. \eqref{eq:defmixeddiscrete}) then we derive an optimal a priori error estimate on the energy norm 
proportional to the approximation properties of the discrete spaces used.
The result is stated in Corollary \ref{corollary:pertubedestimate}.
When $f_3 = 0$ (which is the case for \eqref{eq:NS_Eulercont}) we show that the velocity is exactly divergence free even at the discrete level, i.e.\ that
$ \nabla \cdot u = \nabla \cdot u_h = 0 $ holds pointwise. 
We also show that the scheme is pressure-robust (see Section \ref{Pressurerobustness}). 
This allows deriving an error estimate for the vorticity and velocity independent of the pressure in Theorem \ref{th:pressurerobustestimate}.

The remaining of this paper is divided as follows:
We define the notation used in the paper and discuss some applications of the scheme in Section \ref{Setting}.
We show the well-posedness and derive error estimates for an intermediary problem akin to a Stokes problem in Section \ref{Unpertubedproblem}.
Section \ref{Pertubedproblem} is dedicated to the analysis of problem \eqref{eq:introformulation}.
This is the most technical section.
We derive some additional results in Section \ref{Conservedquantities},
and finally we present a variety of numerical simulations done with our scheme in Section \ref{Numericalsimulations} that validate our results and give some perspectives.
The exterior calculus formalism is introduced in Section \ref{Setting} and heavily used in Section \ref{Unpertubedproblem} and \ref{Pertubedproblem}. 
We assume that the reader has some familiarity with exterior calculus in those two sections (\ref{Unpertubedproblem} and \ref{Pertubedproblem}).
However, no prior knowledge of exterior calculus is expected in Section \ref{Conservedquantities} and \ref{Numericalsimulations}.

\section{Setting} \label{Setting}
The exterior calculus framework allows getting a uniform vision on many objects. 
As such it may appear abstract and confusing at first.
Therefore, we will start by giving some explicit examples of spaces to fix the ideas,
and discuss how to deal with boundary conditions.
Only then we will introduce the full notation and specifications of exterior calculus which will be used in the remaining part of the paper.

Throughout the paper we consider the problem in $3$ dimensions.
However, with usual modifications for the definition of $\omega$,
everything said and proved also works in $2$ dimensions.
This is an advantage of exterior calculus formalism.

\subsection{Function spaces} \label{Spaces}
Our scheme does not rely on a particular choice of discrete spaces, instead we make some assumptions (given in Section \ref{Notations}) on them and any spaces fulfilling these assumptions can be used.
Adequate spaces are readily available on simplicial and cubic meshes (they are given e.g.\ in the 
periodic table of finite elements \cite{periodictable}).
We illustrate here an example of sensible choices. 
Since the discrete spaces depend on the continuous ones (through boundary conditions) 
we begin by setting the continuous spaces for this example.
Let $V^1 = H_0(\mathrm{curl},\Omega)$, 
$V^2 = H_0(\mathrm{div},\Omega)$, $V^3 = L^2(\Omega)$ and $\mathfrak{H}^3 = \mathbb{R} \subset L^2(\Omega)$.
The discrete spaces depend on a polynomial degree $r \in \mathbb{N}$.
Let $\mathbf{T}_h$ be a simplicial triangulation of $\Omega$, we choose the following discretization:
\begin{itemize} \label{defaultspaces}
\item The curl space $V^1_h$ is built upon Nedelec's edge elements of the first kind of degree $r$ (or $P^-_r\Lambda^1$ in the periodic table),
$V_h^1 = \lbrace \omega \in H_0(\mathrm{curl},\Omega);\; \omega_{\vert T} \in P_r^-\Lambda^1(T), \forall T \in \mathbf{T}_h \rbrace$.
\item The velocity space $V^2_h$ is built upon Nedelec's face elements of the first kind of degree $r$ (or $P^-_r\Lambda^2$),
$V_h^2 = \lbrace u \in H_0(\mathrm{div},\Omega);\; u_{\vert T} \in P_r^-\Lambda^2(T), \forall T \in \mathbf{T}_h \rbrace$.
\item The pressure space $V^3_h$ is built upon, discontinuous Galerkin elements of degree $r-1$ (or $P^-_r\Lambda^3$),
$V_h^3 = \lbrace p \in L^2(\Omega);\; p_{\vert T} \in P_r^-\Lambda^3(T), \forall T \in \mathbf{T}_h \rbrace$.
\item The space of discrete harmonic forms $\mathfrak{H}^3_h = \mathbb{R} \subset L^2(\Omega)$ 
is the $L^2$-orthogonal complement of $\mathrm{div}(V^2_h)$ in $V^3_h$.
\end{itemize}
We can replace the first kind with the second, at the expense of increasing the polynomial degrees.
Nedelec elements of first and second kind are respectively the $3$-dimensional equivalent of Raviart-Thomas and of Brezzi-Douglas-Marini elements.
The space $\mathfrak{H}^3_h$ is just the natural way of fixing the pressure (which is defined up to an arbitrary constant when the normal component of the velocity is fixed on the boundary).
Figure \ref{fig:elements} shows the degrees of freedom of this choice of finite elements for $r = 1$ and $r = 2$.
The first element shown on the left corresponds to the space $V^0_h$ that plays no role here.

\begin{figure}
\centering
\includegraphics[width=0.7\textwidth]{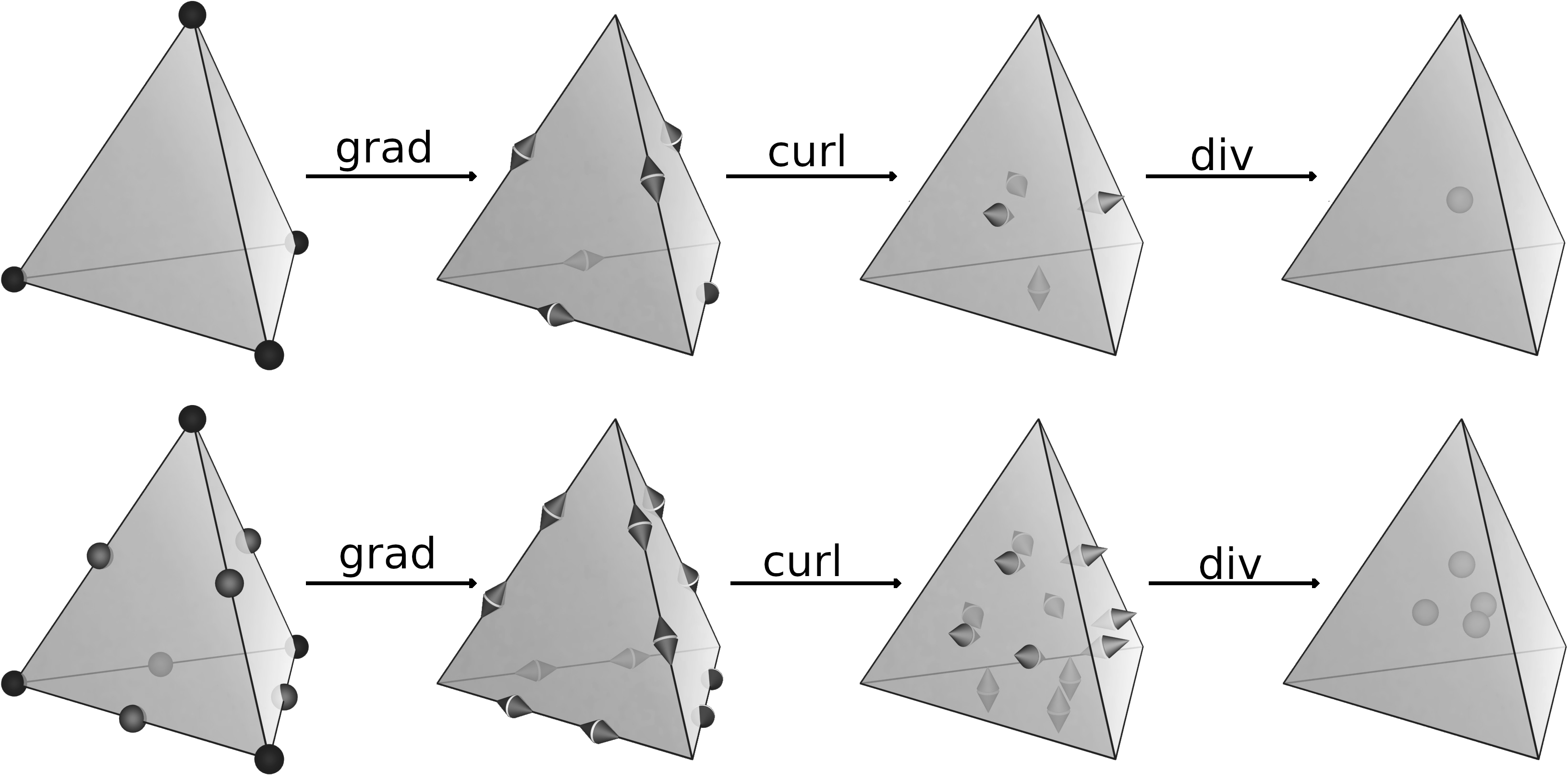}
\caption{Degrees of freedom on reference elements for the polynomial degrees $r = 1$ and $r = 2$.}
\label{fig:elements}
\end{figure}

\begin{remark}
In full generality $\mathfrak{H}^3_h$ is simply the space of the discrete harmonic 3-forms.
If we were to take $V^1 = H(\mathrm{curl},\Omega)$, 
$V^2 = H(\mathrm{div},\Omega)$ and $V^3 = L^2(\Omega)$ instead, we would have $\mathfrak{H}^3 = \mathfrak{H}^3_h = \lbrace 0 \rbrace$.
\end{remark}

For this choice we can make explicit the approximation properties in terms of the size of the mesh $h$ and of the polynomial degree $r$.
For $(\omega,u,p,\phi)$ (resp. $(\omega_h,u_h,p_h,\phi_h)$) the solution of \eqref{eq:defmixedcontinuous} 
(resp. \eqref{eq:defmixeddiscrete}) the estimate of Corollary \ref{corollary:pertubedestimate} reads:
\[ 
\Vert \omega - \omega_h \Vert_{H(\text{curl})} + \Vert u - u_h \Vert_{H(\text{div})} + \Vert p - p_h\Vert_{L^2} + \Vert \phi - \phi_h \Vert_{L^2} = \mathcal{O}(h^{r}) .
\] 

\subsection{Boundary conditions} \label{Boundaryconditions}
We give a quick review of the boundary conditions readily available.
Let $n$ be the unit outward vector normal to the boundary and
$g$ and $h$ arbitrary functions in a suitable space.
Equalities below are always understood $\text{on } \partial \Omega$ and
any combination of the following conditions can be used:
\begin{align}
\omega \times n &=\ g,\  u \cdot n =\ h, \label{eq:bc1} \\
\omega \times n &=\ g,\  p =\ h, \label{eq:bc2} \\
u &=\ g, \label{eq:bc3} \\
u \times n &=\ g,\  p =\ h. \label{eq:bc4}
\end{align}
Condition \eqref{eq:bc3} allows enforcing the no slip condition
with $u \cdot n = 0$, $u \times n = 0$.

In order to implement them, we use essential Dirichlet boundary conditions for: 
\begin{itemize}
\item $\omega \times n = g$ and take test functions in $H_0(\mathrm{curl})$.
\item $u \cdot n = g$ and take test functions in $H_0(\mathrm{div})$.
\end{itemize}
And natural conditions for the other two:
\begin{itemize}
\item $u \times n = g$ by adding $- \int_{\partial \Omega} (g \times \tau) \cdot n ds$ to the left-hand side of equations \eqref{eq:NS_Eulercont}.
\item $p = h$ by adding $\int_{\partial \Omega} h v \cdot n ds$ to the left-hand side of equations \eqref{eq:NS_Eulercont}.
\end{itemize}

\begin{remark}
Although they can all be imposed, only conditions \eqref{eq:bc1} and \eqref{eq:bc4} give a complex.
The following proofs rely on the use of complexes, so we assume that one of these two conditions is chosen.
When we do not use a complex we still have the well-posedness, however the rate of convergence can be impacted.
The case of condition \eqref{eq:bc3}, which is used to impose the no slip condition, 
is studied by Arnold et al.\ \cite{ARNOLD2012,Chen2014}. 
They show that the rate of convergence on the vorticity (for the $H(\text{curl})$ norm) 
and on the pressure (for the $L^2$ norm) is slightly degraded.
However, for a polynomial degree of at least $2$ the rate of convergence of the velocity (for the $H(\text{div})$ norm) is not impacted.
\end{remark}
\subsection{Notation} \label{Notations} 
Finally we state the full framework of exterior calculus which will be used.
We make extensive use of the following notation introduced by D. Arnold \cite{feec-cbms}. 
They are explained in greater detail in the original reference.

\begin{itemize}
\item $d$ is the exterior derivative, $d^*$ its adjoint or codifferential,
\item $W^0 \rightarrow W^1 \rightarrow W^2 \rightarrow W^3$ is the $L^2$ de Rham complex of a bounded domain of $\mathbb{R}^3$.
We shall only use the last three of them.
\item $V^0 \rightarrow V^1 \rightarrow V^2 \rightarrow V^3$ is a dense subcomplex on which the exterior derivative is defined 
(and not just densely defined). Here we have $H^1(\Omega) \rightarrow H(\mathrm{curl},\Omega) \rightarrow H(\mathrm{div},\Omega) \rightarrow L^2(\Omega)$.
\item $V^*_k$ is the domain of the adjoint of $(d,V^k)$.
\item $\Vert \cdot \Vert$ is the $L^2$-norm for scalar or vector valued functions.
\item $\Vert \cdot \Vert_V$ is the $V$ norm, defined by $\Vert \cdot \Vert_V := \Vert \cdot \Vert + \Vert d \cdot \Vert$.
\item Sometimes we take the norm $\Vert \cdot \Vert_{V\cap V^\star}$ on $V \cap V^*$, it is given by $\Vert \cdot \Vert_V + \Vert d^* \cdot \Vert$.
\item $\Vert (u,v) \Vert_{A \times B} := \Vert u \Vert_A + \Vert v \Vert_B$ is the norm on the product space $A \times B$.
\item $V^0_h \rightarrow V^1_h \rightarrow V^2_h \rightarrow V^3_h$ is a discrete subcomplex parametrized by $h$.
\item $P_h$ is the $L^2$-orthogonal projection on the discrete subcomplex, and in general $P_A$ is the $L^2$-orthogonal projection on $A$.
\end{itemize}

We assume that our complexes have the compactness property, 
which means that the inclusion $V^k \cap V_k^* \subset W^k$ is compact for each $k$.
We also assume that there exists a cochain projection $\pi_h^k: V^k \rightarrow V^k_h$, bounded for the $W$-norm, uniformly in $h$.
Each space has the Hodge-decomposition $W^k = \mathfrak{B}^k \operp \mathfrak{H}^k \operp \mathfrak{B}^*_k$,
where $\mathfrak{B}^k := d (V^{k-1})$, $\mathfrak{B}^*_k := d^*(V^*_{k+1})$ 
and $\mathfrak{H}^k := \mathfrak{Z}^k \cap (\mathfrak{B}^k)^\perp$, $\mathfrak{Z}^k$ being the kernel of $d: V^k \rightarrow V^{k+1}$.

In order to measure the approximation properties of the discrete subspaces we introduce the notation $E := E^k$,
defined by:
\begin{equation} \label{eq:approxprop}
\forall k, \forall \sigma \in V^k,\ E^k(\sigma) := \inf_{\tau \in V^k_h} \Vert \sigma - \tau \Vert\ .
\end{equation}

The choice of spaces introduced in Section \ref{defaultspaces} does indeed verify all these properties
with $E = \mathcal{O}(h^r)$ on a dense subset (see \cite{Arnold_2010}).
Each space has a Poincaré inequality (see \cite[Theorem~4.6]{feec-cbms}). 
That is to say, for each $k$, $\exists c_p^k > 0$ such that $\forall z \in \mathfrak{Z}^{k \perp_{V^k}}$, $\Vert z \Vert_{V^k} \leq c_p^k \Vert d z \Vert_{L^2}$.
Moreover this also holds for the discrete subcomplex with constants bounded by $c_p^k\, \Vert \pi_h^k \Vert_{W^k}$ (see \cite[Theorem~5.3]{feec-cbms}).
We set $c_p := \max_{k,h} c_p^k\, \Vert \pi_h^k \Vert_{W^k}$, so that the Poincaré inequality holds for $c_p$ regardless of the space.

For reference let $v_1 \in V^1$, $v_2 \in V^2$, $v_3 \in V^3$. 
The exterior derivative acts as $dv_1 := \text{curl} (v_1)$, $dv_2 := \text{div} (v_2)$ and $dv_3 := 0$.
When it is obvious from the context we shall drop the exponent, 
hence $\pi_h v_1 := \pi_h^1 v_1$, $\pi_h v_2 := \pi_h^2 v_2$.
When we apply an operator (such as $\pi_h$, $d$ or $d^*$) 
to a product space, we mean to apply it to each component (i.e.\ $d (v_1,v_2,v_3) := (d v_1, d v_2, 0)$).
When we add a suffix $_h$ such as $d_h$ instead of $d$ we refer 
to the discrete counterpart of the object.
Since we use a discrete subcomplex, $d_h$ is just the restriction of $d$ to $V_h$.
We will often add a numerical suffix such as $z_2$ for $z \in V^2 \times V^3$,
this means we take the $V^2$ component of $z$.

Usually when dealing with a primal formulation
we will use the variables $(u,p)$, and they can indeed be seen as the velocity and pressure.
However, when we deal with mixed formulations we will frequently write $(u_1,u_2,u_3)$. 
Here $1$,$2$ and $3$ refer to 1-form, 2-form and 3-form and have nothing to do with 
components in a frame of the velocity field.
Specifically we will have the identification $u_1 = \omega$, $u_2 = u$ and $u_3 = p$.

The symbol $A \lesssim B$ means that there exists a constant $C > 0$ independent of $A$ and $B$ 
(depending solely on few specified parameters) such that $A \leq C B$.

\section{Linear steady problem} \label{Unpertubedproblem}
We first study a simpler problem, analogous to a Stokes problem, which is also 
closely related to the Hodge-Laplace problem (see \cite{feec-cbms}).
\begin{definition}
Let $D_0 := \lbrace (u,p) \in (V^2 \cap V^*_2) \times V^*_3 \vert\  d^*u \in V^1 \rbrace$
, $f \in W^2 \times W^3$,
\begin{equation} \label{eq:L0}
L_0 := \begin{bmatrix}
\nu d d^* & - d^*\\
 d & 0 
\end{bmatrix} =
\begin{bmatrix}
\nu \nabla \times \nabla \times &  \nabla\\
\nabla \cdot & 0
\end{bmatrix}.
\end{equation}
\end{definition}
The problem is to find $(u,p) \in P_{(\mathfrak{H}^2 \times \mathfrak{H}^3)^\perp} D_0$, such that
\begin{equation} \label{eq:linearsteadypnotation}
\forall (v,q) \in P_{(\mathfrak{H}^2 \times \mathfrak{H}^3)^\perp} D_0,\ \langle L_0 (u,p),(v,q) \rangle = \langle f,(v,q)\rangle. 
\end{equation}
In vector calculus notation we want $(u,p)$ such that 
\[\nu \nabla \times (\nabla \times u) + \nabla p = f_2,\quad \nabla \cdot u = f_3. \]
\subsection{Continuous Well-posedness}
Now we introduce the mixed formulation.
Recall that $u_2$ (resp. $u_3$) defined below corresponds to $u$ (resp. $p$) in \eqref{eq:linearsteadypnotation}.
The mixed formulation is characterized by the introduction of the auxiliary variable $u_1$ corresponding to $\nabla \times u_2$ or $d^* u_2$.
We define $B_0$ by:
\begin{equation}
\begin{aligned}
B_0&((u_1,u_2,u_3,\phi_2,\phi_3),(v_1,v_2,v_3,\chi_2,\chi_3)) := \\ 
&\langle u_1,v_1\rangle  - \langle u_2,d v_1\rangle
+ \nu \langle  d u_1,v_2\rangle  - \langle u_3, d v_2\rangle + \langle \phi_2,v_2\rangle \\
&\ + \langle d u_2,v_3\rangle+ \langle \phi_3,v_3\rangle 
+ \langle u_2,\chi_2\rangle  + \langle u_3,\chi_3\rangle.\\
\end{aligned}
\end{equation}
The problem reads: 
Given $(f_2,f_3) \in W^2 \times W^3$,
find $(u_1,u_2,u_3,\phi_2,\phi_3) \in V^1 \times V^2 \times V^3 \times \mathfrak{H}^2 \times \mathfrak{H}^3 $ such that $\forall (v_1,v_2,v_3,\chi_2,\chi_3) \in V^1 \times V^2 \times V^3 \times \mathfrak{H}^2 \times \mathfrak{H}^3$,
\begin{equation}
\begin{aligned} \label{eq:pbB0}
B_0((u_1,u_2,u_3,\phi_2,\phi_3),(v_1,v_2,v_3,\chi_2,\chi_3)) =& 
\ \langle f_2,v_2\rangle + \langle f_3, v_3 \rangle.
\end{aligned}
\end{equation}

\begin{remark} \label{rem:removedstar}
We removed $d^*$ from the formulation to prepare for the discrete formulation \eqref{eq:discretelinearsteady}.
Indeed $d^*$ and $d_h^*$ are barely related.
In fact $d_h^*$ is a global operator which would greatly deteriorate the sparsity pattern of the system.
\end{remark}

\begin{lemma} \label{lemma:discretewp1}
There is $\alpha > 0$, such that $\forall (u_1,u_2,u_3,\phi_2,\phi_3) \in V^1 \times V^2 \times V^3 \times \mathfrak{H}^2 \times \mathfrak{H}^3$
\[ 
\begin{aligned}
\sup_{(v_1,v_2,v_3,\chi_2,\chi_3)} &\frac{B_0((u_1,u_2,u_3,\phi_2,\phi_3),(v_1,v_2,v_3,\chi_2,\chi_3))}{\Vert (v_1,v_2,v_3,\chi_2,\chi_3) \Vert_V} 
\geq \alpha \Vert (u_1,u_2,u_3,\phi_2,\phi_3) \Vert_V.
\end{aligned}
\] 
\end{lemma}

\begin{proof}
For any $(u_1,u_2,u_3,\phi_2,\phi_3) \in V^1 \times V^2 \times V^3 \times \mathfrak{H}^2 \times \mathfrak{H}^3$, 
let $\rho_1 \in \mathfrak{B}^{*}_1 \cap V^1$ be such that $d \rho_1 = P_{\mathfrak{B}} u_2$,
and $\rho_2 \in \mathfrak{B}^{*}_2 \cap V^2$ be such that $d \rho_2 = P_{\mathfrak{B}} u_3 + d u_2$.
And take $v_1 = \nu u_1 - \frac{\nu}{c_p^2} \rho_1$,
$v_2 = P_{\mathfrak{B}} u_2 + d u_1 - \rho_2 + \phi_2$,
$v_3 = d u_2 - P_{\mathfrak{B}} u_3 + \phi_3$,
$\chi_2 = P_{\mathfrak{H}} u_2$,
$\chi_3 = P_{\mathfrak{H}} u_3$.
The Poincaré inequality gives:
\begin{equation} \label{eq:Poincare_rho1}
\Vert \rho_1 \Vert_{V^1} \leq c_p \Vert P_{\mathfrak{B}} u_2 \Vert,\quad
\Vert \rho_2 \Vert_{V^2} \leq c_p \Vert P_{\mathfrak{B}} u_3 \Vert + c_p \Vert d u_2 \Vert. 
\end{equation}
And we easily see that, for hidden constants depending only on $\nu$ and $c_p$,
\[ 
\Vert v_1 \Vert_{V^1} + \Vert v_2 \Vert_{V^2} + \Vert v_3 \Vert_{V^3} + \Vert \chi_2 \Vert + \Vert \chi_3 \Vert \lesssim
\Vert u_1 \Vert_{V^1} + \Vert u_2 \Vert_{V^2} + \Vert u_3 \Vert_{V^3} + \Vert \phi_2 \Vert + \Vert \phi_3 \Vert.
\] 
Using the orthogonality of the Hodge decomposition we get:
\begin{equation} \label{eq:B0_wp}
\begin{aligned}
&B_0((u_1,u_2,u_3,\phi_2,\phi_3),(v_1,v_2,v_3,\chi_2,\chi_3))\\
&\,=\langle u_1,\nu u_1 - \frac{\nu}{c_p^2} \rho_1\rangle  - \langle u_2,d (\nu u_1 - \frac{\nu}{c_p^2} \rho_1)\rangle 
+ \nu \langle  d u_1,P_{\mathfrak{B}} u_2 + d u_1 - \rho_2 + \phi_2\rangle \\
&\quad - \langle u_3, d (P_{\mathfrak{B}} u_2 + d u_1 - \rho_2 + \phi_2)\rangle  
+ \langle d u_2,d u_2 - P_{\mathfrak{B}} u_3 + \phi_3\rangle + \langle u_2,P_{\mathfrak{H}} u_2\rangle \\
&\quad + \langle u_3,P_{\mathfrak{H}} u_3\rangle
+ \langle \phi_2,P_{\mathfrak{B}} u_2 + d u_1 - \rho_2 + \phi_2\rangle  + \langle \phi_3,d u_2 - P_{\mathfrak{B}} u_3 + \phi_3\rangle\\
&\,= \nu \langle u_1, u_1\rangle  - \frac{\nu}{c_p^2} \langle u_1, \rho_1\rangle  - \nu \langle u_2,d u_1\rangle  + \frac{\nu}{c_p^2}\langle u_2, P_{\mathfrak{B}} u_2\rangle 
+ \nu \langle  d u_1,u_2\rangle \\
&\quad + \nu \langle d u_1,d u_1\rangle
+ \langle u_3,P_{\mathfrak{B}} u_3\rangle + \langle u_3, d u_2 \rangle  
+ \langle d u_2,d u_2 \rangle - \langle d u_2, u_3 \rangle \\
&\quad+ \langle u_2,P_{\mathfrak{H}} u_2\rangle 
+ \langle u_3,P_{\mathfrak{H}} u_3\rangle 
+ \langle \phi_2, \phi_2\rangle  + \langle \phi_3,\phi_3\rangle \\
&\,= \nu \Vert u_1 \Vert ^2 - \frac{\nu}{c_p^2} \langle u_1, \rho_1\rangle 
+ \frac{\nu}{c_p^2} \Vert P_{\mathfrak{B}} u_2 \Vert ^2
+ \nu \Vert d u_1 \Vert ^2 + \Vert P_{\mathfrak{B}} u_3 \Vert ^2
+ \Vert d u_2 \Vert ^2 \\
&\quad+ \Vert P_{\mathfrak{H}} u_2 \Vert^2
+ \Vert P_{\mathfrak{H}} u_3 \Vert ^2
+ \Vert \phi_2 \Vert ^2 + \Vert \phi_3 \Vert ^2 . 
\end{aligned}
\end{equation}
Applying \eqref{eq:Poincare_rho1} and $\langle u_1,\rho_1\rangle  \leq \Vert u_1 \Vert \Vert \rho_1 \Vert \leq \frac{c_p^2}{2} \Vert u_1 \Vert ^2 + \frac{1}{2 c_p^2} \Vert \rho_1 \Vert ^2$
to \eqref{eq:B0_wp} yields:
\[
\begin{aligned}
&B_0((u_1,u_2,u_3,\phi_2,\phi_3),(v_1,v_2,v_3,\chi_2,\chi_3)) \geq 
\frac{\nu}{2} \Vert u_1 \Vert ^2 + \frac{\nu}{2 c_p^2} \Vert P_{\mathfrak{B}} u_2 \Vert ^2\\
&\quad\quad+ \nu \Vert d u_1 \Vert ^2 + \Vert u_3 \Vert ^2
+ \Vert d u_2 \Vert ^2
+ \Vert P_{\mathfrak{H}} u_2 \Vert^2 + \Vert P_{\mathfrak{H}} u_3 \Vert ^2
+ \Vert \phi_2 \Vert ^2 + \Vert \phi_3 \Vert ^2 . 
\end{aligned}
\]

Finally, $d u_3 = 0$, $d \phi_2 = d \phi_3 = 0$ and $\Vert P_{\mathfrak{B}^*} u_2 \Vert \leq c_p \Vert d P_{\mathfrak{B}^*} u_2 \Vert = c_p \Vert d u_2 \Vert $, hence
\[B_0((u_1,u_2,u_3,\phi_2,\phi_3),(v_1,v_2,v_3,\chi_2,\chi_3)) \gtrsim
\Vert (u_1,u_2,u_3,\phi_2,\phi_3) \Vert ^2 _V, \]
where the hidden constant depends only on $c_p$ and $\nu$.
\end{proof}

\begin{lemma} \label{lemma:discretewp2}
Given any $(v_1,v_2,v_3,\chi_2,\chi_3) \in V^1 \times V^2 \times V^3 \times \mathfrak{H}^2 \times \mathfrak{H}^3$ there is 
$(u_1,u_2,u_3,\phi_2,\phi_3) \in V^1 \times V^2 \times V^3 \times \mathfrak{H}^2 \times \mathfrak{H}^3$ such that 
\[B_0((u_1,u_2,u_3,\phi_2,\phi_3),(v_1,v_2,v_3,\chi_2,\chi_3)) > 0\ .
\]
\end{lemma}
\begin{proof}
If $P_{\mathfrak{B}} v_2 \neq 0$ take $u_1 \in \mathfrak{B}^*_1$ such that $d u_1 = P_\mathfrak{B} v_2$, $\phi_2 = 0$, $u_3 = \phi_3 = 0$. 
Then if $\langle u_1,v_1 \rangle = 0$ take $u_2 = 0$ else take $u_2 = d v_1 \frac{\langle u_1,v_1 \rangle}{\langle d v_1, d v_1 \rangle}$ 
($d v_1 \neq 0$ since $P_{\mathfrak{B}^*} v_1 \neq 0$ as $\langle u_1,v_1 \rangle \neq 0$).
If $P_\mathfrak{B} v_2 = 0$, simply take $u_1 = v_1$,
$u_2$ such that $d u_2 = P_{\mathfrak{B}} v_3, P_{\mathfrak{B}} u_2 = - d v_1, P_{\mathfrak{H}} u_2 = \chi_2$ (this is possible by the Hodge decomposition),
$u_3$ such that $P_{\mathfrak{B}} u_3 = - d v_2, P_{\mathfrak{H}} u_3 = \chi_3$,
$\phi_2 = P_{\mathfrak{H}} v_2$ and
$\phi_3 = P_{\mathfrak{H}} v_3$.
\end{proof}

Lemma \ref{lemma:discretewp1} and Lemma \ref{lemma:discretewp2} 
together with the 
Babuška–Lax–Milgram theorem
give the well-posedness of \eqref{eq:pbB0}.
Moreover, we have for $c$ depending only on $c_p$ and $\nu$:
\begin{equation}
\Vert u_1 \Vert_{V^1} + \Vert u_2 \Vert_{V^2} + \Vert u_3 \Vert_{V^3} + \Vert \phi_2 \Vert_{V^2} 
+ \Vert \phi_3 \Vert_{V^3} \leq c\, \Vert f \Vert.
\end{equation}
\subsection{Discrete problem and error estimate} 
The discrete problem reads: \newline
Given $(f_2,f_3) \in W^2 \times W^3$,
find $(u_{1h},u_{2h},u_{3h},\phi_{2h},\phi_{3h}) \in V^1_h \times V^2_h \times V^3_h \times \mathfrak{H}^2_h \times \mathfrak{H}^3_h $ such that 
$\forall (v_{1h},v_{2h},v_{3h},\chi_{2h},\chi_{3h}) \in V^1_h \times V^2_h \times V^3_h \times \mathfrak{H}^2_h \times \mathfrak{H}^3_h$,
\begin{equation} \label{eq:discretelinearsteady}
B_0((u_{1h},u_{2h},u_{3h},\phi_{2h},\phi_{3h}),(v_{1h},v_{2h},v_{3h},\chi_{2h},\chi_{3h})) = \langle f_2,v_{2h}\rangle + \langle f_3,v_{3h} \rangle.
\end{equation}
\begin{lemma}
Problem \eqref{eq:discretelinearsteady} is well-posed.
\end{lemma}
\begin{proof}
Since we have a discrete Poincaré inequality (see \cite{feec-cbms}) and use a subcomplex 
we can apply exactly the same proof as in the continuous case. 
\end{proof}

Next we derive a basic error estimate for a global norm.
Improved error estimates for each component are given by Theorem \ref{th:improvedestimates}.
First we define: 
\begin{equation} \label{eq:mu}
\mu := \max_{k \in \lbrace 2,3 \rbrace} \sup_{r \in \mathfrak{H}^k, \Vert r \Vert = 1} \Vert (I - \pi_h^k) r \Vert.
\end{equation}
\begin{theorem} \label{th:unpertubedestimate}
Given $(f_2,f_3) \in W^2 \times W^3$,
let $(u_1,u_2,u_3,\phi_2,\phi_3)$ be the solution of the continuous problem \eqref{eq:pbB0} and
$(u_{1h},u_{2h},u_{3h},\phi_{2h},\phi_{3h})$ the solution of the discrete problem \eqref{eq:discretelinearsteady},
then for $E$ given by \eqref{eq:approxprop} it holds:
\[
\begin{aligned}
\Vert& (u_1 - u_{1h},u_2 - u_{2h},u_3 - u_{3h}) \Vert_V + \Vert (\phi_2 - \phi_{2h},\phi_3 - \phi_{3h}) \Vert \\ 
& \lesssim \inf_{v_1 \in V^1_h} \Vert u_1 - v_1 \Vert_{V^1} + \inf_{v_2 \in V^2_h} \Vert u_2 - v_2 \Vert_{V^2}  + E(u_3) + E(\phi_2) + E(\phi_3) \\
& \ + \mu (E(P_\mathfrak{B} u_2) + E(P_\mathfrak{B} u_3)).
\end{aligned}
\]
\end{theorem}
\begin{proof}
For all $(v_{1h},v_{2h},v_{3h},\chi_{2h},\chi_{3h}) \in V^1_h \times V^2_h \times V^3_h \times \mathfrak{H}^2_h \times \mathfrak{H}^3_h$
we have:
\[
\begin{aligned}
B_0((u_1,u_2,u_3,\phi_2,\phi_3),(v_{1h},&v_{2h},v_{3h},\chi_{2h},\chi_{3h})) \\
&= \langle f_2,v_{2h}\rangle + \langle f_3,v_{3h}\rangle + 
\langle u_2,\chi_{2h}\rangle + \langle u_3,\chi_{3h}\rangle\ .
\end{aligned}
\]
\pagebreak

Let $(v_{1},v_{2},v_{3},\chi_{2},\chi_{3})$ be the $V$-orthogonal projection 
of $(u_{1},u_{2},u_{3},\phi_{2},\phi_{3})$ into their respective discrete spaces.
By the continuity of $B_0$ it holds,\\
$\forall (v_{1h},v_{2h},v_{3h},\chi_{2h},\chi_{3h}) \in V^1_h \times V^2_h \times V^3_h \times \mathfrak{H}^2_h \times \mathfrak{H}^3_h$,
\[
\begin{aligned}
&B_0((u_{1h}-v_1,u_{2h}-v_2,u_{3h}-v_3,\phi_{2h}-\chi_2,\phi_{3h}-\chi_3),
(v_{1h},v_{2h},v_{3h},\chi_{2h},\chi_{3h})) \\
&\quad= B_0((u_{1}-v_1,u_{2}-v_2,u_{3}-v_3,\phi_{2}-\chi_2,\phi_{3}-\chi_3) ,(v_{1h},v_{2h},v_{3h},\chi_{2h},\chi_{3h}))\\
&\quad\quad\ - \langle u_2,\chi_{2h}\rangle - \langle u_3,\chi_{3h}\rangle \\
&\quad= B_0((u_{1}-v_1,u_{2}-v_2,u_{3}-v_3,\phi_{2}-\chi_2,\phi_{3}-\chi_3) ,(v_{1h},v_{2h},v_{3h},\chi_{2h},\chi_{3h}))\\
&\quad\quad\ - \langle P_{\mathfrak{H}_h} u_2,\chi_{2h}\rangle - \langle P_{\mathfrak{H}_h} u_3,\chi_{3h}\rangle \\
&\quad\lesssim (\Vert u_{1}-v_1 \Vert_{V^1} + \Vert u_{2}-v_2\Vert_{V^2} +\Vert u_{3}-v_3\Vert + \Vert \phi_{2}-\chi_2\Vert + \Vert \phi_{3}-\chi_3\Vert\\
&\quad\quad\ + \Vert P_{\mathfrak{H}_h} u_2 \Vert_{V^2}+ \Vert P_{\mathfrak{H}_h} u_3 \Vert_{V^3})
(\Vert v_{1h} \Vert_{V^1} + \Vert v_{2h} \Vert_{V^2} + \Vert v_{3h} \Vert + \Vert \chi_{2h}\Vert + \Vert \chi_{3h}\Vert) ,
\end{aligned}
\]
where the hidden constant depends only on $\nu$.

From the discrete inf-sup condition we have 
\begin{equation*}
\begin{aligned}
(\Vert u_{1h}-v_1 \Vert_{V^1} &+ \Vert u_{2h}-v_2\Vert_{V^2} +\Vert u_{3h}-v_3\Vert + \Vert \phi_{2h}-\chi_2\Vert + \Vert \phi_{3h}-\chi_3\Vert) \\
&\lesssim (\Vert u_{1}-v_1 \Vert_{V^1} + \Vert u_{2}-v_2\Vert_{V^2} +\Vert u_{3}-v_3\Vert + \Vert \phi_{2}-\chi_2\Vert\\
&\ \ \ + \Vert \phi_{3}-\chi_3\Vert + \Vert P_{\mathfrak{H}_h} u_2 \Vert_{V^2}+ \Vert P_{\mathfrak{H}_h} u_3 \Vert),
\end{aligned}
\end{equation*}
where the hidden constant depends only on $\nu$ and on the discrete constant of Poincaré.
The theorem is inferred from \cite[Theorem 5.2]{feec-cbms} and \cite[Equation~(33)]{Arnold_2010}
which state: 
\[ \Vert P_{\mathfrak{H}_h} u_i \Vert_V \lesssim\, \mu\, E(P_\mathfrak{B} u_i), \quad
\Vert \phi_{i}-\chi_i\Vert \lesssim\,  E(\phi_{i}),
\quad \forall i \in \lbrace 2,3 \rbrace. 
\]
\end{proof}
\begin{remark}
$E(\phi_i)$ is understood as viewing $\phi_i$ as an element of $V^i \supset \mathfrak{H}^i$,
\[ E(\phi_i) := \inf_{q \in V^i_h} \Vert \phi_i - q \Vert. 
\]
\end{remark}

\section{Linearized problem} \label{Pertubedproblem} 
We construct our scheme by adding some lower order terms to Problem \eqref{eq:pbB0}.
We recall the correspondence with the names used for variables in the introduction: $(u_1,u_2,u_3,u_p) \equiv (\omega,u,p,\phi)$
in \eqref{eq:introformulation}.
To keep the notation bearable we shall also write $u_{\mathfrak{B}} := P_{\mathfrak{B}} u$, $u_{\mathfrak{H}} := P_{\mathfrak{H}} u$ and so on.

We consider two linear maps  $l_3: W^1 \rightarrow W^2$ and $l_5: W^2 \rightarrow W^2$ 
(we chose these names to match those used by Arnold \& Li \cite{Arnold2016}).
Define $D := \lbrace (u,p) \in (V^2 \cap V^*_2) \times (V^*_3 \cap \mathfrak{H}^{3\perp}) \vert \ d^* u \in V^1 \rbrace$
, $W = W^2 \times (W^3 \cap \mathfrak{H}^{3 \perp})$ and
\begin{equation} \label{eq:Llambda}
L := \begin{bmatrix}
(\nu d + l_3) d^* + l_5 & - d^*\\
d & 0 
\end{bmatrix}.
\end{equation}
We consider the primal problem:
Given $f \in W$, find $(u,p) \in D$ such that 
\begin{equation} \label{eq:defmixedprimal}
L (u,p) = f.
\end{equation}
We also define the dual operator $L'$ on $D' := \lbrace (u,p) \in (V^2 \cap V^*_2) \times (V^*_3 \cap \mathfrak{H}^{3\perp}) \vert\  (\nu d^* + l_3^*) u \in V^1 \rbrace$ by:
\begin{equation} \label{eq:Llambdap}
L' := \begin{bmatrix}
d (\nu d^* + l_3^*) + l_5^* & d^*\\
- d & 0 
\end{bmatrix}.
\end{equation}

As an intermediary step, we wish to extend $L$ to a larger domain and introduce:
\begin{equation*}
\begin{gathered}
\widebar{L_\lambda}: (V^2 \cap V^*_2) \times (V^*_3 \cap \mathfrak{H}^{3\perp}) \rightarrow (V^*_2 \times (W^3 \cap \mathfrak{H}^{3 \perp}))', \\
\widebar{L_\lambda}(u,p)(v,q) := \langle \nu d^* u,d^* v \rangle + \langle l_3 d^* u + l_5 u - \lambda d^* p,v \rangle + \langle \lambda d u,q \rangle,
\end{gathered}
\end{equation*}
\pagebreak
\begin{equation*}
\begin{gathered}
\widebar{L_\lambda'}: (V^2 \cap V^*_2) \times (V^*_3 \cap \mathfrak{H}^{3\perp}) \rightarrow (V^*_2 \times (W^3 \cap \mathfrak{H}^{3 \perp}))',\\
\widebar{L_\lambda'}(u,p)(v,q) := \langle \nu d^* u + l_3^* u, d^* v \rangle + \langle l_5^* u 
+ \lambda d^* p,v \rangle + \langle -\lambda d u,q \rangle. 
\end{gathered}
\end{equation*}
Here $\lambda$ is a positive parameter introduced to simplify the proof of Theorem \ref{th:Liso}.
In Theorem \ref{th:Liso} we shall see that they are almost always isomorphisms.
We define the solution operator $K := (\widebar{L_1'})^{-1}$, and we assume that 
\begin{equation} \label{eq:regularity35}
d^* (\widebar{L_1}^{-1})_2 (W) \subset V^1,\ (\nu d^* + l_3^*) (K)_2 (W) \subset V^1, 
\end{equation}
where $(\widebar{L_1}^{-1})_2$ and $(K)_2$ are the projections on the first component of the product space taken after the operators.
Moreover, we assume that $\Vert d d^* (K)_2 \Vert_{W \rightarrow W^2} $ and $\Vert d l_3^* (K)_2 \Vert_{W \rightarrow W^2}$ are bounded.
We show in Section \ref{Regularityassumptions} that these assumptions are satisfied when $l_3$ and $l_5$ are those used in our scheme.

The proof follows the same outline as \cite{Arnold2016}.
First we prove that the continuous primal formulation gives an isomorphism,
then we prove that the continuous mixed formulation is well-posed.
Lastly we prove the well-posedness of the discrete mixed formulation
and give an estimation of the error in energy norm.
\subsection{Continuous primal formulation}
\begin{theorem} \label{th:Liso}
$\widebar{L_\lambda} + \gamma \langle \cdot , \cdot \rangle$ is a bounded isomorphism for all $\gamma \in \mathbb{C}$ except for a countable subset.
\end{theorem}

\begin{proof}
Let $c = max(\Vert l_3 \Vert, \Vert l_5 \Vert)$,
for $(u,p) \in (V^2 \cap V^*_2) \times (V^*_3 \cap \mathfrak{H}^{3\perp})$
take 
$v_\mathfrak{B} = u_\mathfrak{B}$,
$v_\mathfrak{H} = u_\mathfrak{H}$,
$v_{\mathfrak{B}^*} = - d^* p$,
$q = d u$. We have:
\begin{equation} \label{eq:Lextendproof}
\begin{aligned}
(\gamma \langle \cdot , \cdot \rangle + \widebar{L_\lambda})(u,p)(v,q) 
=&\ \nu \langle d^* u,d^*u\rangle  + \lambda \langle d^*p,d^*p\rangle  + \lambda \langle d u,d u\rangle  \\
&\  + \langle l_3 d^* u,u_\mathfrak{B} - d^* p + u_\mathfrak{H}\rangle  + \langle l_5 u, u_\mathfrak{B} - d^* p + u_{\mathfrak{H}}\rangle  \\
&\  + \gamma \langle u,u_\mathfrak{B} + u_\mathfrak{H}\rangle  - \gamma \langle u,d^* p\rangle  + \gamma\langle p,d u\rangle  \\
=&\ \nu \Vert d^* u\Vert^ 2+ \lambda \Vert d^* p||^2 + \lambda \Vert d u \Vert ^2 +  \gamma \Vert u_\mathfrak{B} + u_\mathfrak{H} \Vert ^2 \\
&\  + \langle l_3 d^* u,u_\mathfrak{B} -d^*p +u_\mathfrak{H} \rangle  + \langle l_5 u,u_\mathfrak{B} - d^* p + u_\mathfrak{H}\rangle. 
\end{aligned}
\end{equation}
We bound the last line from \eqref{eq:Lextendproof} with the Cauchy-Schwarz inequality:
\[
\begin{aligned}
&\vert \langle l_3 d^* u,u_\mathfrak{B} -d^*p +u_\mathfrak{H} \rangle \vert + \vert \langle l_5 u,u_\mathfrak{B} - d^* p + u_\mathfrak{H}\rangle \vert\\ 
&\quad \leq \frac{c^2}{2\nu} (\Vert u_\mathfrak{B} + u_\mathfrak{H} \Vert ^2 +  \Vert d^* p \Vert ^2)
+ \frac{\nu}{2} \Vert d^* u \Vert ^2  
 + \frac{c^2}{2}(\Vert u_\mathfrak{B} + u_\mathfrak{H} \Vert^2 + \Vert d^* p \Vert ^2) + \frac{1}{2} \Vert u \Vert^2.
\end{aligned}
\]
Since $ \Vert u_{\mathfrak{B}^*} \Vert ^2 \leq c_p^2 \Vert d u \Vert ^2$ 
and $ \Vert u \Vert ^2 = \Vert u_{\mathfrak{B}^*} \Vert ^2 + \Vert u_\mathfrak{B} + u_\mathfrak{H} \Vert^2$
we have:
\[
\begin{aligned}
(\gamma \langle \cdot , \cdot \rangle + \widebar{L_\lambda})(u,p)(v,q) \geq &  
\ (\gamma - \frac{c^2}{2 \nu} - \frac{c^2}{2} - \frac{1}{2}) \Vert u_\mathfrak{B} + u_\mathfrak{H} \Vert ^2 + \frac{\nu}{2} \Vert d^* u\Vert^2\\ 
&+ (\lambda - \frac{c^2}{2 \nu} - \frac{c^2}{2}) \Vert d^* p||^2 
+ (\lambda - \frac{c_p^2}{2}) \Vert d u \Vert ^2  .
\end{aligned}
\]

We use the Poincaré inequality to bound $\Vert u_{\mathfrak{B}^*} \Vert$ by $\Vert d u \Vert$
and $\Vert p_{\mathfrak{B}} \Vert$ by $\Vert d^* p \Vert$ (on the dual complex).
Since $\Vert p_{\mathfrak{B}} \Vert$ = $\Vert p \Vert$ (as $p \in V^3 \cap \mathfrak{H}^{3 \perp}$)
we have for $\lambda$ and $\gamma$ large enough (solely depending on $c$, $\nu$ and $c_p$):
\[
\begin{aligned}
(\gamma \langle \cdot , \cdot \rangle + \widebar{L_\lambda})(u,p)(v,q) &\gtrsim \Vert u \Vert ^2 + \Vert d u \Vert ^2 
+ \Vert d^* u \Vert ^2 + \Vert p \Vert ^2 + \Vert d^* p \Vert ^2 \\
&\gtrsim \Vert (u,p) \Vert _{V^2 \cap V^*_2 \times V^*_3} ^2.
\end{aligned}
\]

Clearly $\Vert (v,q) \Vert_{V^*_2 \times W^3} \leq \Vert (u,p) \Vert _{V^2 \cap V^*_2 \times V^*_3}$ 
as $\Vert d^* v \Vert = \Vert d^* u_\mathfrak{B} \Vert \leq \Vert d^* u \Vert$ and
$(\gamma \langle \cdot , \cdot \rangle + \widebar{L_\lambda})$ is continuous as a bilinear form from $(V^2 \cap V^*_2 \times V^*_3) \times (V^*_2 \times (W^3 \cap \mathfrak{H}^{3 \perp}))$\pagebreak.

The only thing left to show in order to use the Babuška–Lax–Milgram theorem is the second condition.
For any $(v,q) \neq 0$ we must find $(u,p)$ such that 
$(\gamma \langle \cdot , \cdot \rangle + \widebar{L_\lambda})(u,p)(v,q) > 0$.
We can take $\lambda$ and $\gamma$ such that $\lambda ^2 = \gamma \nu$.
We consider two cases: \newline 
When $v = 0$ take $u$ solution of the Hodge-Dirac problem (see \cite{Leopardi2016}): $d u = q$, $d^* u = 0$.
We readily check that $(\gamma \langle \cdot , \cdot \rangle + \widebar{L_\lambda})(u,0)(0,q) = \lambda \langle d u,q\rangle  = \lambda \langle q,q\rangle  > 0$. 
When $v \neq 0$, we can find $u \in \lbrace w \in V^2 \cap V^*_2 \vert d^* w \in V^1, dw \in V^*_3 \rbrace$ such that $((\nu d + l_3) d^* + l_5 + \gamma + \nu d^* d) u = v$ 
(see \cite{Arnold2016}).
Take $p = - \nu/\lambda d u$,
then since $\lambda = \gamma \nu / \lambda$ it holds:
\begin{equation*}
\begin{aligned}
(\gamma \langle \cdot , \cdot \rangle + \widebar{L_\lambda})(u,p)(v,q)  &= \langle \nu d d^* u,v\rangle  + \langle (l_3 d^* + l_5 + \gamma)u,v\rangle  - \lambda \langle d^* p,v\rangle \\ 
&\ + \lambda \langle d u,q\rangle  + \gamma \langle p,q \rangle \\
&= \langle ((\nu d + l_3) d^* + l_5 + \gamma + \nu d^* d) u, v \rangle \\
&\ + \lambda \langle d u,q \rangle - \frac{\gamma \nu}{\lambda} \langle du, q \rangle \\
&= \langle v,v\rangle.
\end{aligned}
\end{equation*}

Thus $(\gamma \langle \cdot , \cdot \rangle + \widebar{L_\lambda})$ is a bounded isomorphism from 
$(V^2 \cap V^*_2) \times (V^*_3 \cap \mathfrak{H}^{3 \perp})$ to $(V_2^* \times (W^3 \cap \mathfrak{H}^{3 \perp}))'$.
Since $I: (V^2 \cap V^*_2) \times (V^*_3 \cap \mathfrak{H}_3^\perp) \rightarrow (V^*_2 \times (W^3 \cap \mathfrak{H}^{3 \perp}))'$ 
is compact by the compactness property,
$I(\gamma \langle \cdot , \cdot \rangle + \widebar{L_\lambda})^{-1}$ is also compact. Since the spectrum of a compact operator is at most countable, 
$Id + \eta I (\gamma \langle \cdot , \cdot \rangle + \widebar{L_\lambda})^{-1}$ has a bounded inverse for all $\eta \in \mathbb{C}$
except for a countable subset.
Therefore, by composing to the right with $\gamma \langle \cdot,\cdot \rangle + \widebar{L_\lambda}$ we get that $\widebar{L_\lambda} + (\gamma + \eta) I$ has almost always a bounded inverse.
\end{proof}

Hence, up to an arbitrary small perturbation, $\widebar{L_\lambda}$ is a bounded isomorphism 
from $(V^2 \cap V^*_2) \times (V^*_3 \cap \mathfrak{H}_3^\perp) $ to $(V^*_2 \times (W^3 \cap \mathfrak{H}^{3 \perp}))'$.

\begin{remark}
We could have left $\mathfrak{H}^3$ in the domain 
and the proof above would still work.
However, in this case, $\widebar{L_\lambda}$ would never have
been an isomorphism since its image cannot reach $\mathfrak{H}^3$.
\end{remark}

We have the same result for the dual problem.
\begin{lemma}
For all $\gamma \in \mathbb{C}$ except for a countable subset, $\widebar{L_\lambda'} + \gamma \langle \cdot , \cdot \rangle$ is a bounded isomorphism.
\end{lemma}
\begin{proof}
The same proof as the one of Theorem \ref{th:Liso} works.
The only differences will be a sign in the chosen $(v,q)$,
$l_5$ and $l_3$ replaced with $l_5^*$ and $l_3^*$ 
and $(l_3 d^* u,v)$ changed to $(l_3^* u,d^* v)$.
This does not add any difficulty in the proof.
\end{proof}

\begin{remark}
The proof Theorem \ref{th:Liso} requires taking $\lambda$ to be sufficiently large,
however for any $f_2 \in (V_2^{*})'$, $f_3 \in (W^3 \cap \mathfrak{H}^{3 \perp})'$, $\lambda_0 > 0$, $\lambda_1 > 0$ we have the following equivalence:
\[ \widebar{L_{\lambda_0}} (u,p) = (f_2,f_3) \Leftrightarrow \widebar{L_{\lambda_1}} (u, \frac{\lambda_0}{\lambda_1} p) = (f_2, \frac{\lambda_1}{\lambda_0} f_3)\ .
\]
Therefore, if $\widebar{L_{\lambda_0}}$ is a bounded isomorphism for a given $\lambda_0$, it easily follows that $\widebar{L_{\lambda_1}}$ is a bounded isomorphism for any $\lambda_1 > 0$,
in particular for $\lambda_1 = 1$.
The same argument works for $\widebar{L_{\lambda}'}$.
\end{remark}

From here onward, we assume that $\widebar{L_1}$ and $\widebar{L_1'}$ are bounded isomorphisms. 

\subsection{Well-posedness of the continuous mixed formulation}
As we did in the unperturbed case, 
we introduce an auxiliary variable in the problem.
In the following, $\mathbf{u}$ is a shortcut for $(u_1,u_2,u_3,u_p)$.
We define $B$ by:
\begin{equation} \label{eq:defBlambda}
\begin{aligned}
B(\mathbf{u},\mathbf{v}) :=&\  \langle u_1,v_1\rangle  - \langle u_2,d v_1\rangle 
 + \nu \langle  d u_1,v_2\rangle  - \langle u_3, d v_2\rangle  + \langle d u_2,v_3\rangle  \\
&\ + \langle l_3 u_1,v_2\rangle  + \langle l_5 u_2,v_2\rangle  + \langle u_p,v_3\rangle  + \langle u_3,v_p\rangle. 
\end{aligned}
\end{equation}
The mixed formulation is:
Given $(f_2,f_3) \in W$, find $\mathbf{u} := (u_1,u_2,u_3,u_p) \in V^1 \times V^2 \times V^3 \times \mathfrak{H}^3 $ such that 
$\forall \mathbf{v} := (v_1,v_2,v_3,v_p) \in V^1 \times V^2 \times V^3 \times \mathfrak{H}^3$,
\begin{equation} \label{eq:defmixedcontinuous}
B(\mathbf{u},\mathbf{v}) = \langle f_2,v_2\rangle  + \langle f_3,v_3\rangle.
\end{equation}

For $(u,p)$ solution of \eqref{eq:defmixedprimal}, 
it immediately appears that 
$(d^* u,u,p,0)$ solves \eqref{eq:defmixedcontinuous}.
Now for $(u_1,u_2,u_3,u_p)$ solution of \eqref{eq:defmixedcontinuous},
the first line implies that $u_2 \in V^*_2$, $d^* u_2 = u_1$ and therefore $d^* u_2 \in V^1$.
The second line implies that $u_3 \in V^*_3$.
And the last implies that $u_3 \perp \mathfrak{H}^3$.
Therefore $(u_2,u_3) \in D$, and it obviously solves \eqref{eq:defmixedprimal}.

\begin{theorem}
Under the regularity assumption \eqref{eq:regularity35}, Problem \eqref{eq:defmixedcontinuous} is well-posed.
\end{theorem}
The proof follows from the Babuška–Lax–Milgram theorem along with Lemma \ref{lemma:pertubedwp2} and inf-sup condition \eqref{eq:pertubedwp1}.
In the following, hidden constants only depend on $\Vert l_3 \Vert$, $\Vert l_5 \Vert$, $\Vert K \Vert$, $\Vert d (d^* + l_3^*) (K)_2 \Vert$, $\nu$ and on constants of Poincaré.
We shall write $V = V^1 \times V^2 \times V^3 \times \mathfrak{H}^3$.

\begin{lemma} \label{lemma:defz}
For all $(0,u_2,u_3,0) = \mathbf{u} \in V$, there exists $\mathbf{z} \in V$ such that $\Vert \mathbf{z} \Vert_V \lesssim \Vert \mathbf{u} \Vert_V$
and 
$\forall \mathbf{\omega} \in V, B(\mathbf{\omega},\mathbf{z}) = \langle \mathbf{\omega},\mathbf{u}\rangle $.
\end{lemma}
\begin{proof}
Take $(z_2,z_3) = K (u_2,P_{\mathfrak{H}^\perp} u_3)$ and $\xi = - (\nu d^* + l_3^*)z_2$.
Then $\xi \in V^1$ by assumption \eqref{eq:regularity35}, and
by the definition of $K := (\widebar{L_1'})^{-1}$ we have $\forall \mathbf{\omega} \in V$: 
\begin{equation} \label{eq:lemmadefz1}
\begin{aligned}
- \langle d z_2,\omega_3 \rangle =& \langle P_{\mathfrak{H}^\perp} u_3, \omega_3 \rangle, \\
\langle d (\nu d^* + l_3^*) z_2 + l_5^* z_2 + d^* z_3, \omega_2 \rangle =& \langle u_2,\omega_2 \rangle. 
\end{aligned}
\end{equation}
We set $\mathbf{z} := ( \xi,z_2,z_3,u_{3 \mathfrak{H}})$. Applying \eqref{eq:lemmadefz1} we have:
\begin{equation*}
\begin{aligned}
B( \mathbf{\omega},\mathbf{z}) 
=&\, \langle \omega_1, \xi\rangle  - \langle \omega_2,d \xi\rangle
 + \langle ( \nu d + l_3)  \omega_1,z_2\rangle  - \langle \omega_3,d z_2\rangle  + \langle l_5 \omega_2,z_2\rangle  \\
& +\langle d \omega_2, z_3\rangle  + \langle \omega_3,u_{3 \mathfrak{H}}\rangle  + \langle \omega_p,z_3\rangle  \\
=&\, \langle \omega_1, \xi\rangle  - \langle \omega_2,d \xi\rangle
 + \langle \omega_1, (\nu d^* + l_3^*) z_2\rangle  + \langle \omega_3,P_{\mathfrak{H}^\perp} u_3 \rangle  + \langle \omega_2,l_5^* z_2\rangle  \\
& + \langle \omega_2, d^* z_3\rangle  + \langle \omega_3,u_{3 \mathfrak{H}}\rangle  \\
=&\, \langle \omega_1, \xi\rangle  - \langle \omega_2,d \xi\rangle
 + \langle \omega_1,(\nu d^* + l_3^*) z_2 \rangle - \langle \omega_2,d (\nu d^* + l_3^*)z_2\rangle  \\
& + \langle \omega_2,u_2\rangle + \langle \omega_3, P_{\mathfrak{H}^\perp} u_3\rangle  + \langle \omega_3,u_{3 \mathfrak{H}}\rangle  \\
=& -\langle \omega_1,( \nu d^* + l_3^*)  z_2\rangle  + \langle \omega_2,d ( \nu d^* + l_3^*)  z_2\rangle  \\
& + \langle \omega_1,(\nu d^* + l_3^*) z_2\rangle  - \langle \omega_2,d (\nu d^* + l_3^*)z_2\rangle 
 + \langle \omega_3,u_3\rangle  + \langle \omega_2,u_2\rangle  \\
=&\, \langle \omega_2,u_2\rangle  + \langle \omega_3,u_3\rangle. 
\end{aligned}
\end{equation*}
Moreover, since $\widebar{L_1'}$ is a bounded isomorphism, so is $K$ thus
\begin{equation*}
\begin{aligned}
\Vert \mathbf{z} \Vert_V &\leq \Vert \nu d^* z_2 \Vert + \Vert l_3^* z_2 \Vert + \Vert z_2 \Vert_{V^2} + \Vert z_3 \Vert + \Vert d \xi \Vert + \Vert u_{3\mathfrak{H}} \Vert \\
&\lesssim \Vert z_2 \Vert_{V^2 \cap V^*_2} + \Vert z_3 \Vert_{V^*_3} + \Vert d \xi \Vert + \Vert u_{3 \mathfrak{H}} \Vert \\
&\lesssim (\Vert K \Vert + \Vert d(d ^* + l_3^*) K \Vert + 1) \Vert \mathbf{u} \Vert \lesssim \Vert \mathbf{u} \Vert_V.
\end{aligned}
\end{equation*}
\end{proof}

\begin{lemma} \label{lemma:continuous2}
For all $\mathbf{u} \in V$, there exists $\mathbf{z} \in V$ such that $\Vert \mathbf{z} \Vert_V \lesssim \Vert \mathbf{u} \Vert_V$
and $B(\mathbf{u},\mathbf{z}) \gtrsim \Vert d u_1 \Vert ^2 + \Vert d u_2 \Vert^2 + \Vert u_1 \Vert ^2 + \Vert u_p \Vert ^2 - \Vert u_2 \Vert^2$.
\end{lemma}
\begin{proof}
Let $c = \max(\Vert l_3 \Vert,\Vert l_5 \Vert)$. 
We begin with some preliminary computations:
\begin{equation} \label{eq:mixedeq3}
\begin{aligned}
B(\mathbf{u},(0,0,u_p,0)) = \langle u_p,u_p\rangle ,
\end{aligned}
\end{equation}

\begin{equation} \label{eq:mixedeq2}
\begin{aligned}
B(\mathbf{u},(0,d u_1,0,0)) =&\ \nu \langle d u_1,d u_1\rangle  + \langle l_3 u_1,d u_1\rangle  + \langle l_5 u_2,d u_1\rangle  \\
\geq& \frac{1}{2} \nu \Vert d u_1 \Vert ^2 - \frac{c^2}{\nu} ( \Vert u_1 \Vert^2 + \Vert u_2 \Vert ^2 ),
\end{aligned}
\end{equation}

\begin{equation} \label{eq:mixedeq1}
\begin{aligned}
B(\mathbf{u},(\nu u_1,u_{2 \mathfrak{B}},d u_2,0)) =&\  \nu \langle u_1,u_1 \rangle - \nu \langle u_2,d u_1 \rangle + \nu \langle  d u_1,u_{2\mathfrak{B}}\rangle\\
&- \langle u_3, d u_{2\mathfrak{B}}\rangle  + \langle d u_2,d u_2\rangle + \langle l_3 u_1,u_{2 \mathfrak{B}}\rangle\\
&+ \langle l_5 u_2,u_{2 \mathfrak{B}}\rangle  +
\langle u_p,d u_2\rangle  + \langle u_3,0\rangle  \\
\geq& \frac{\nu}{2} \Vert u_1 \Vert ^2 + \Vert d u_2 \Vert ^2  - \left (\frac{c^2}{2 \nu} + c \right ) \Vert u_2 \Vert ^2.
\end{aligned}
\end{equation}
Clearly it is possible to construct a suitable $\mathbf{z}$ from a linear combination of \eqref{eq:mixedeq3}, \eqref{eq:mixedeq2} and \eqref{eq:mixedeq1}.
Bounds on norms are easily checked, for example:
\[\Vert (0, d u_1,0,0) \Vert_V = \Vert d u_1 \Vert + \Vert d d u_1 \Vert = \Vert d u_1 \Vert \lesssim \Vert \mathbf{u} \Vert_V. 
\]
\end{proof}

Combining the two preceding lemmas gives:
\begin{equation} \label{eq:pertubedwp1}
\forall \mathbf{u} \in V, \sup_{\Vert \mathbf{v} \Vert_V = 1} \vert B(\mathbf{u},\mathbf{v}) \vert \gtrsim \Vert \mathbf{u} \Vert_V\ .
\end{equation}

\begin{lemma} \label{lemma:pertubedwp2}
For any $\mathbf{v} \in V$, there is $\mathbf{u} \in V$ such that $B(\mathbf{u},\mathbf{v}) > 0$.
\end{lemma}
\begin{proof}
Given $\mathbf{v} \neq 0 \in V$, if $v_2 = 0$, $v_3 = 0$ and $v_p = 0$
take $\mathbf{u} = (v_1,0,0,0)$, then 
\[B(\mathbf{u},\mathbf{v}) = \langle v_1,v_1\rangle  > 0.
\]
Else take $(u_2,u_3) = \widebar{L_1}^{-1}(v_2,P_{\mathfrak{H}^\perp}v_3) + (0,v_p)$,
$u_p = P_\mathfrak{H} v_3$ and $u_1 = d^* u_2$ ($u_1 \in V^1$ by assumption \eqref{eq:regularity35})
then \[B(\mathbf{u},\mathbf{v}) = \langle v_2,v_2\rangle  + \langle v_3,v_3\rangle  + \langle q,q\rangle  > 0.
\]
\end{proof}

\subsection{Discrete well-posedness}\label{Discretewellposedness}

We introduce the notation $V_h = V_h^1 \times V_h^2 \times V_h^3 \times \mathfrak{H}^3_h$.
The discrete variational problem is the same as the continuous one, 
replacing $V$ by $V_h$.
Hence we shall still use the notation $B$, this time as a function 
from $V_h \times V_h$ to $\mathbb{R}$.
So that the discrete problem is:
Given $(f_2,f_3) \in W$, find $\mathbf{u}_h := (u_{1h},u_{2h},u_{3h},u_{ph}) \in V^1_h \times V^2_h \times V^3_h \times \mathfrak{H}^3_h $ such that 
$\forall \mathbf{v}_h := (v_{1h},v_{2h},v_{3h},v_{ph}) \in V^1_h \times V^2_h \times V^3_h \times \mathfrak{H}^3_h$,
\begin{equation} \label{eq:defmixeddiscrete}
B(\mathbf{u}_h,\mathbf{v}_h) = \langle f_2,v_{2h}\rangle  + \langle f_3,v_{3h}\rangle.
\end{equation}

Considering the dual of problem \eqref{eq:linearsteadypnotation} with $\nu = 1$,
we have: 
$D_0 := \lbrace (u,p) \in (V^2 \cap V^*_2) \times V^3_* \vert\  d^*u \in V^1 \rbrace$ and
$L_0'(u,p) := (dd^* u + d^*p,-du)$.
Let $K_0$ be the solution operator of the dual problem.
We have $K_0 = (L_0')^{-1}$ when $L_0'$ is viewed as an isomorphism from $P_{\mathfrak{H}^\perp} D_0$ to $P_{\mathfrak{H}^\perp} (W^2 \times W^3)$, and $K_0$ is extended by $0$ on $\mathfrak{H}$.
Explicitly we have the decomposition: $\forall (f_2,f_3) \in W^2 \times W^3$,
\begin{equation*} 
 (f_2,f_3) = (d d^* (K_0)_2(f_2,f_3) + d^* (K_0)_3(f_2,f_3),-d (K_0)_2(f_2,f_3)) + (P_{\mathfrak{H}} f_2,P_{\mathfrak{H}} f_3)
\end{equation*}
and a similar expression for their discrete counterparts.
Therefore, 
\begin{equation} \label{eq:decompK}
\forall (z_2,z_3) \in D_0,\ (P_{\mathfrak{H}^\perp} z_2,P_{\mathfrak{H}^\perp} z_3) = L_0' K_0 (z_2,z_3) = K_0 L_0' (z_2,z_3).
\end{equation}

As mentioned before this problem is closely related to the one studied by Arnold  \& Li \cite{Arnold2016}.
Since the mixed variable part is almost unchanged we shall
use the generalized canonical projection $\Pi_h$ (see \cite{Arnold2016}) and
we state its properties below.

\begin{lemma} Under the condition of \cite[Theorem~5.1]{Arnold2016}: \label{lemma:PIh}
\begin{itemize}
\item $\Pi_h$ is a projection uniformly bounded in the V-norm.
\item $d \Pi_h = P_{\mathfrak{B}_h} d$.
\item $\forall w \in V^k, \Vert (I - \Pi_h)w \Vert \lesssim \Vert (I - \pi_h)w \Vert + \eta_0' \Vert d w \Vert$.
\item $\forall w,v \in V^k, \vert \langle(I - \Pi_h)w,v\rangle \vert \lesssim  (\Vert (I - \pi_h) w \Vert + \eta_0' \Vert d w \Vert)(\Vert (I - \pi_h) v \Vert + \eta_0' \Vert d v \Vert) + \alpha_0' \Vert dw \Vert \Vert dv \Vert$.
\end{itemize}
\end{lemma}
Here $\eta_0', \alpha_0' \to 0$ when $h \to 0$. They are given, along the proof in the reference \cite{Arnold2016}.

\begin{definition} \label{def:estimates}
We shall use the following notation in this section:
\[ \delta_0 := \Vert (I - \pi_h) K_0 \Vert,\ \mu_0 := \Vert (I - \pi_h) P_\mathfrak{H} \Vert, 
\]
\[ \eta_0 := \max \lbrace \Vert (I - \pi_h) d K_0 \Vert,\Vert (I - \pi_h) d^* (K_0)_2 \Vert \rbrace, 
\]
\[ \alpha_0 := \eta_0(1 + \eta_0) + \mu_0 + \delta_0 + \mu_0 \delta_0 + \eta_0', 
\]
\[ \eta := \max \lbrace \delta_0, \mu_0, \eta_0,\Vert (I - \pi_h) l_3^* (K)_2 \Vert,\Vert (I - \pi_h)dl_3^* (K)_2 \Vert\rbrace. 
\]
\end{definition}

\begin{lemma} \label{lemma:estimatesforzh}
The following bounds hold: 
\[\Vert K_0 - K_{0h}P_h \Vert \lesssim \alpha_0, 
\]
\[ \Vert d K_0 - d K_{0h} P_h \Vert + \Vert d^* (K_0)_2 - d_h^* (K_{0h})_2 P_h \Vert \lesssim \eta_0. 
\]
\end{lemma}
\begin{proof}
Let $(f_2,f_3) \in (W^2 \times W^3)$.
The idea is to apply the error estimate of Theorem \ref{th:unpertubedestimate} for $(u_2,u_3) = K_0 (f_2,f_3)$,
$(\phi_2,\phi_3) = P_{\mathfrak{H}}(f_2,f_3)$, $u_1 = d^* u_2$, $(u_{2h},u_{3h}) = K_{0h}P_h (f_2,f_3)$,
$(\phi_{2h},\phi_{3h}) = P_{\mathfrak{H}_h} P_h (f_2,f_3)$, $u_{1h} = d_h^* u_{2h}$. Unfortunately we cannot conclude with the crude estimate of Theorem \ref{th:unpertubedestimate}  
because of the error on $u_1$.
We need improved estimates that give
\begin{equation*}
\begin{aligned}
\Vert u_2 - u_{2h} \Vert + \Vert u_3 - u_{3h} \Vert &\lesssim  (1 + \mu_0) E(u_2) + E(u_3) + \eta_0 E(u_1)\\
&\+ (\eta_0^2 + \delta_0 + \eta_0') E(d u_1) + \eta_0' E(\phi_2) + E(d u_2),\\
\Vert d u_2 - d u_{2h} \Vert + \Vert u_1 - u_{1h} \Vert &\lesssim E(d u_2) + E(u_1) + \eta_0 E(d u_1).
\end{aligned}
\end{equation*}
We conclude since $E(u_2) + E(u_3) \leq \delta_0 \Vert (f_2,f_3) \Vert$, $E(u_1) + E(d u_2) \leq \eta_0 \Vert (f_2,f_3) \Vert$ and 
$E(d u_1) + E(\phi_2) \leq \Vert (f_2,f_3) \Vert$, the last coming from $d u_1 = P_\mathfrak{B} f_2$.
These proofs are technical and mostly follow those of \cite[Theorem 3.11]{Arnold_2010},
see Appendix \ref{Improvederrorestimates}.
\end{proof}

\begin{lemma} \label{lemma:erroronharmonics}

For $f \in \mathfrak{H}^\perp$ we have $\Vert P_{\mathfrak{H}_h} f \Vert \lesssim \mu_0 \Vert f \Vert$.
\end{lemma}
\begin{proof}
We recall the mixed formulation for the Hodge-Laplacian problem.
The bilinear form is given by:
\[B((\sigma,u,p),(\tau,v,q)) = \langle \sigma,\tau\rangle  - \langle u,d \tau\rangle  + \langle d \sigma,v\rangle  + \langle d u,d v\rangle  + \langle p,v\rangle  + \langle u,q\rangle .
\]
In the continuous case the bilinear form acts on $(V^1 \times V^2 \times \mathfrak{H}^2)^2$
and on $(V^1_h \times V^2_h \times \mathfrak{H}_h^2)^2$ in the discrete case.
Let $(\sigma,u,p) \in V^1 \times V^2 \times \mathfrak{H}^2$ be such that $\forall (\tau,v,q) \in V^1 \times V^2 \times \mathfrak{H}^2$, $B((\sigma,u,p),(\tau,v,q)) = (f,v)$,
and let $(\sigma_h,u_h,p_h) \in V^1_h \times V^2_h \times \mathfrak{H}^2_h$ be such that $\forall (\tau,v,q) \in V^1_h \times V^2_h \times \mathfrak{H}^2_h$, $B((\sigma_h,u_h,p_h),(\tau,v,q)) = (f,v)$.
Then \cite[Theorem~5.6 p.~62]{feec-cbms} gives the error estimate $\Vert p - p_h \Vert \lesssim E(p) + \mu_0 E(d \sigma)$.
We have $P_{\mathfrak{H}} f = p = 0$, $P_{\mathfrak{H}_h} f = p_h$ and $E(p) = 0$. 
The well-posedness of the Hodge-Laplacian problem gives:
\[\Vert P_{\mathfrak{H}_h} f \Vert = \Vert p - p_h \Vert \lesssim 0 + \mu_0 E(d \sigma) \lesssim \mu_0 \Vert f \Vert.
\]
\end{proof}

\begin{theorem} \label{th:estimatezh}
For $z = (z_2,z_3) \in D_0$, let $z_h = (z_{2h},z_{3h}) := K_{0h} P_h L_0' z + P_{\mathfrak{H}_h} P_\mathfrak{H} z$. It holds:
\[ \Vert z - z_h \Vert \lesssim \alpha_0 \Vert L_0' z\Vert,\ \Vert d (z - z_h) \Vert + \Vert d^* z_2 - d^*_h z_{2h} \Vert \lesssim \eta_0 \Vert L_0' z \Vert, 
\]
\[ \Vert P_h (d d^* z_2 + d^* z_3) - (d d^*_h z_{2h} + d^*_h z_{3h}) \Vert \leq \mu_0 \Vert L_0' z \Vert.
\]
\end{theorem}
\begin{proof}
The same proof as \cite[Theorem~5.2]{Arnold2016} works.
Starting from
\begin{equation*}
\begin{aligned}
z - z_h &= (P_{\mathfrak{H}^\perp} z - P_{\mathfrak{H}_h^\perp} z_h) +
(P_{\mathfrak{H}} z - P_{\mathfrak{H}_h} z_h) \\
&= ( K_0 - K_{0h} P_h) L_0' z + (I - P_{\mathfrak{H}_h})P_\mathfrak{H} z,
\end{aligned}
\end{equation*}
$P_{\mathfrak{H}_h} P_{\mathfrak{H}} = P_{\mathfrak{Z}_h} P_{\mathfrak{H}}$ and $\pi_h \mathfrak{Z} \subset \mathfrak{Z}_h$ we infer:
\[\Vert (I - P_{\mathfrak{H}_h})P_\mathfrak{H} z\Vert \leq \Vert (I - \pi_h) P_\mathfrak{H} z \Vert \leq \mu_0 \Vert z \Vert \lesssim \alpha_0 \Vert L_0' z \Vert. 
\]
The first two estimates follow from Lemma \ref{lemma:estimatesforzh}.
For the last estimate, \eqref{eq:decompK} gives 
\begin{equation*}
d d_h^* z_{2h} + d_h^* z_{3h} = (L_{0h}')_2 z_h = (L_{0h}' K_{0h} P_h)_2 L_0' z = (P_{\mathfrak{B}_h} + P_{\mathfrak{B}^*_h}) P_h (L_0')_2 z,
\end{equation*}
\begin{equation*}
\begin{aligned}
\Vert P_h (d d^* z_2 + d^* z_3) - (d d^*_h z_{2h} + d^*_h z_{3h}) \Vert &= 
\Vert (I - (P_{\mathfrak{B}_h} + P_{\mathfrak{B}^*_h})) P_h (L_0')_2 z \Vert \\
& = \Vert P_{\mathfrak{H}_h} (L_0')_2 z \Vert.
\end{aligned}
\end{equation*}
We conclude with Lemma \ref{lemma:erroronharmonics}.
\end{proof} 

Given $\mathbf{u} \in V_h$,
we define $g := (u_2,P_{\mathfrak{H}^\perp}u_3)$, $z: = Kg$, $\xi := -(d^* + l_3^*)z_2$ and $\mathbf{z} := (\xi,z,P_\mathfrak{H} u_3)$.
\begin{theorem} \label{th:discreteproj}
There is $\mathbf{z}_h \in V_h$ such that $\forall \mathbf{\omega} \in V_h$,
$\Vert \mathbf{z}_h \Vert_V \lesssim  \Vert \mathbf{z} \Vert_V$ uniformly in h and
$\vert B(\mathbf{\omega},\mathbf{z} - \mathbf{z}_h) \vert \leq \epsilon_h \Vert \mathbf{\omega} \Vert_V \Vert \mathbf{u} \Vert$,
where $\epsilon_h \rightarrow 0$ as $h \rightarrow 0$.
\end{theorem}
\begin{proof}
Let $z_h = K_{0h} P_h L_0' z + P_{\mathfrak{H}_h} P_\mathfrak{H} z$,
$\xi_h = - d_h^* z_{2h} - \Pi_h l_3^* z_2$.
Theorem \ref{th:estimatezh} gives
\begin{equation}
\begin{gathered}
\Vert z - z_h \Vert \lesssim \alpha_0 \Vert g \Vert,\ \Vert d(z - z_h) \Vert \lesssim \eta_0 \Vert g \Vert,\  
\Vert d^* z_2 - d^*_h z_{2h} \Vert \lesssim \eta_0 \Vert g \Vert,\\
\Vert \xi - \xi_h\Vert \leq \Vert d^* z_2 - d^*_h z_{2h} \Vert + \Vert (I - \Pi_h) l_3^* z_2 \Vert.
\end{gathered}
\end{equation}
Lemma \ref{lemma:PIh} and the boundedness of $d\,l^*_3\,K$ give:
\[ \Vert (I - \Pi_h) l_3^* z_2 \Vert \lesssim \Vert (I - \pi_h)l^*_3 z_2 \Vert + \eta_0' \Vert dl^*_3z_2 \Vert \lesssim (\eta + \eta_0') \Vert g \Vert. 
\]
Finally, since $\forall \omega_2 \in V^2_h \subset V^2$
$\langle d \omega_2, z_3 \rangle = \langle \omega_2, d^* z_3 \rangle$
and $\langle d \omega_2, z_{3h} \rangle = \langle \omega_2, d^*_h z_{3h} \rangle$,
we have $\forall \mathbf{\omega} \in V_h$:
\begin{equation*}
\begin{aligned}
\vert B(\mathbf{\omega},&(\xi - \xi_h, z_2 - z_{2h}, z_3 - z_{3h},P_\mathfrak{H} u_3 - P_{\mathfrak{H}_h}P_\mathfrak{H}  u_3))\vert \\
&= \vert \langle \omega_1,\xi - \xi_h\rangle
 - \langle \omega_2,d(\xi - \xi_h)\rangle  + \langle (\nu d + l_3)\omega_1,z_2 - z_{2h}\rangle \\
&\quad - \langle \omega_3,d(z_2 - z_{2h})\rangle 
 + \langle l_5 \omega_2,z_2 - z_{2h}\rangle \\
&\quad + \langle d\omega_2, z_3 - z_{3h}\rangle
 + \langle \omega_p,z_3 - z_{3h}\rangle + \langle \omega_3,P_\mathfrak{H} u_3 
 - P_{\mathfrak{H}_h}P_\mathfrak{H} u_3 \rangle \vert\\
&\lesssim  \Vert \mathbf{\omega} \Vert_V (\Vert \xi - \xi_h \Vert 
 + 2 \Vert z_2 - z_{2h} \Vert + \Vert d(z_2 - z_{2h}) \Vert \\ 
&\quad + \Vert z_3 - z_{3h} \Vert + \Vert (I - P_{\mathfrak{H}_h}) P_\mathfrak{H} u_3 \Vert)  
+ \vert  \langle \omega_2,P_h d^* z_3 - d_h^* z_{3h} - P_h d(\xi - \xi_h)\rangle \vert \\
&\lesssim  \Vert \mathbf{\omega} \Vert_V \left[ (\eta + 2 \eta_0
+ 2 \alpha_0 + \eta_0' + \mu_0) \Vert \mathbf{u} \Vert + \Vert P_h d^* z_3 - d_h^* z_{3h} - P_h d( \xi - \xi_h) \Vert \right].
\end{aligned}
\end{equation*}
Since $\eta$, $\eta_0$, $\eta_0'$, $\mu_0$ and $\alpha_0$ all converge toward $0$ when $h \rightarrow 0$,
the only thing left to prove is that 
$\Vert P_h d^* z_3 - d_h^* z_{3h} - P_h d( \xi - \xi_h) \Vert \lesssim \epsilon \Vert \mathbf{u} \Vert$ 
where $\epsilon \rightarrow 0$ when $h \rightarrow 0$.
\pagebreak\\
To do so we start from Theorem \ref{th:estimatezh} and expand:
\[\Vert P_h d^* z_3 - d_h^*z_{3h} + P_h d d^* z_2 - d d^*_h z_{2h} \Vert \lesssim \mu_0 \Vert g \Vert,
\]
\[ -d (\xi - \xi_h) = d (d^* + l_3^*) z_2 - d d^*_h z_{2h} - d \Pi_h l^*_3 z_2 = d d^* z_2 - d d^*_h z_{2h} + d(I - \Pi_h) l^*_3 z_2 .
\]
We conclude with Lemma \ref{lemma:PIh} since we can find a bounded cochain projection $\pi_h$ 
such that $\pi_h d = P_{\mathfrak{B}_h} d$ (see \cite[Theorem~3.7]{Arnold_2010}) so
\[ \Vert d(I - \Pi_h) l_3^* z_2 \Vert = \Vert (I - P_{\mathfrak{B}_h}) d l^*_3 z_2 \Vert \lesssim \Vert (I - \pi_h) d l^*_3 z_2 \Vert \leq \eta \Vert g \Vert. 
\]
\end{proof}

\begin{lemma} \label{lemma:defzh}
For all $\mathbf{u} \in V_h$ and $\mathbf{z} \in V$ defined in Theorem \ref{th:discreteproj},
there exists $c  > 0$ independent of $h$ and $\mathbf{\sigma} \in V_h$ such that $\Vert \mathbf{\sigma} \Vert_V \lesssim \Vert \mathbf{u} \Vert_V$
and $B(\mathbf{u},\mathbf{z} + \mathbf{\sigma}) \geq c \Vert \mathbf{u}\Vert_V^2 $.
\end{lemma}
\begin{proof}
Starting from Lemma \ref{lemma:defz}, we construct $\mathbf{\sigma}$ in the same way as we did in Lemma \ref{lemma:continuous2} in the continuous case.
We just have to correct the harmonic part adding
\[B(\mathbf{\omega},(0,0,-P_{\mathfrak{H}_h} z_3,0)) = - \langle \omega_p,z_3\rangle. 
\]
\end{proof}

\begin{theorem} \label{th:wellposednesspertubeddiscrete}
There are two positive constants $h_0$ and $C_0$ 
such that for all $h \in (0,h_0]$,
there exists a unique $\mathbf{u} \in V_h$ such that $\forall \mathbf{v} \in V_h$,
$B(\mathbf{u},\mathbf{v}) = (\mathbf{f},\mathbf{v})$.
Moreover we have $\Vert \mathbf{u} \Vert_V \leq C_0 \Vert \mathbf{f} \Vert$.
\end{theorem}
\begin{proof}
For $\mathbf{u} \in V_h$, Theorem \ref{th:discreteproj} and Lemma \ref{lemma:defzh} give
$\mathbf{z} \in V$, 
$\mathbf{z}_h, \mathfrak{\sigma} \in V_h$ 
with $\Vert \mathbf{\sigma} \Vert_V \lesssim \Vert \mathbf{u} \Vert_V$ 
and two constants $b$ and $c$ independent of $h$
such that:
\[ \vert B(\mathbf{u},\mathbf{z} + \mathbf{\sigma}) \vert \geq c \Vert \mathbf{u} \Vert_V ^2, \quad 
B(\mathbf{u},\mathbf{z}_h - \mathbf{z}) \leq \epsilon_h b \Vert \mathbf{u} \Vert_V ^2. 
\]
Combining the two with a triangle inequality readily gives:
\begin{equation*}
\begin{aligned}
\vert B(\mathbf{u},\mathbf{\sigma} + \mathbf{z}_h) \vert =&\ \vert B(\mathbf{u},\mathbf{z + \sigma}) + B(\mathbf{u},\mathbf{z}_h - \mathbf{z}) \vert \\
\geq&\ \vert B(\mathbf{u},\mathbf{z} + \mathbf{\sigma}) \vert - \vert B(\mathbf{u},\mathbf{z}_h - \mathbf{z}) \vert \\
\geq&\ c \Vert \mathbf{u} \Vert_V^2 - \epsilon_h b \Vert \mathbf{u} \Vert_V^2
\geq (c - \epsilon_h b) \Vert \mathbf{u} \Vert^2_V.
\end{aligned}
\end{equation*}
Since $\epsilon_h \rightarrow 0$ as $h \rightarrow 0$ 
we can find $h_0$ such that $\forall h \in (0,h_0]$,
$c - \epsilon_h b \geq c - \epsilon_{h_0} b > 0$.

By Theorem \ref{th:discreteproj} and from the expression of $\mathbf{\sigma}$ we find:
\[\Vert \mathbf{\sigma} + \mathbf{z}_h \Vert_V \lesssim \Vert \mathbf{u} \Vert_V + \Vert \mathbf{z} \Vert_V \lesssim \Vert \mathbf{u} \Vert_V.
\]
This ends the proof since $V_h$ is of finite dimension.
\end{proof}

\begin{corollary} \label{corollary:pertubedestimate}
If assumption \eqref{eq:regularity35} holds then for $h \leq h_0$ given by Theorem \ref{th:wellposednesspertubeddiscrete},
and for $\mathbf{u}$ (resp. $\mathbf{u}_h$) the solution of the continuous problem \eqref{eq:defmixedcontinuous} (resp. of the discrete problem \eqref{eq:defmixeddiscrete}) 
it holds:
\begin{equation*}
\begin{aligned}
\Vert \mathbf{u}& - \mathbf{u}_h \Vert_V \lesssim \\
& \inf_{v_1 \in V^1_h} \Vert u_1 - v_1 \Vert_{V^1} + \inf_{v_2 \in V^2_h} \Vert u_2 - v_2 \Vert_{V^2} + E(u_3) + E(u_p) + \mu_0 E(P_\mathfrak{B} u_3).
\end{aligned}
\end{equation*}
\end{corollary}
\begin{proof}
The same proof as the one of Theorem \ref{th:unpertubedestimate} works.
\end{proof}

\begin{remark}
The hidden constant of Corollary \ref{corollary:pertubedestimate} depends on $\Vert l_5 \Vert$ which, in the case of problem \eqref{eq:NS_Eulercont}, blows up when $\delta t \to 0$. 
A more subtle analysis is required to make explicit the dependency of the error on $\delta t$.
For a single time step in the setting of Corollary \ref{corollary:pertubedestimate} the error will actually decrease when $\delta t \to 0$.
We prove an estimate for the error on $u_2$ in a very general setting in Theorem \ref{th:timeindeest}.
\end{remark}

\section{Conserved quantities} \label{Conservedquantities}
Lastly we prove that our scheme does indeed verify the properties mentioned in the introduction as well as the regularity assumption 
\pagebreak
\eqref{eq:regularity35}.

\subsection{Regularity assumptions} \label{Regularityassumptions}
Problem \eqref{eq:NS_Eulercont} is a special case of Problem \eqref{eq:defmixedcontinuous} taking suitable $l_3$ and $l_5$.
We prove below that assumption \eqref{eq:regularity35} is valid if $u^{n-1} \in H^2(\Omega)$ and 
if the domain is smooth enough to have $H_0(\mathrm{curl},\Omega) \cap H(\mathrm{div},\Omega) \subset H^1(\Omega)$, as discussed in \cite[Chapter~3.2]{GiraultRaviart}.
We use the notation $H^1(\Omega)$ both for scalar and for vector fields.
First we need the following lemma:
\begin{lemma} \label{lemma:sobolev}
If $B \in H^1(\Omega)$ and $A \in H^2(\Omega)$ then $A \times B \in H^1(\Omega)$ for a smooth enough domain $\Omega$ of $\mathbb{R}^3$ (or $\mathbb{R}^2$).
\end{lemma}
\begin{proof}
We have $H^1(\Omega) \subset L^4(\Omega)$ and $H^2(\Omega) \subset C^{0}(\widebar{\Omega})$
by Sobolev Embedding theorems
thus $\forall i,j,k \in \lbrace x,y,z\rbrace$, $\partial_i A_j \in L^4(\Omega)$, $\partial_i B_j \in L^2(\Omega)$.
Terms of the form $A_i\, \partial_j B_k$ 
are the product of a bounded function with a function in $L^2(\Omega)$ 
and those of the form $B_i\, \partial_j A_k$ are the product of two functions in $L^4(\Omega)$.
\end{proof}
Going back to problem \eqref{eq:NS_Eulercont}, we take
$V^1 = H(\mathrm{curl},\Omega)$, $V^2 = H(\mathrm{div},\Omega)$, $V^3 = L^2(\Omega)$ and $\mathfrak{H}^3 = {0}$.
Regarding assumption \eqref{eq:regularity35},
$K$ is a bounded isomorphism from $L^2 \times L^2$ to $(H(\mathrm{div}) \cap H_0(\mathrm{curl})) \times H^1_0 \subset H^1 \times H^1$.
If we assume $u \in H^2$
and $v \in H^1$ then using the scalar triple product we have $\forall \sigma \in H(\mathrm{curl})$: 
\[2 \langle l_3 \sigma, v \rangle = \int (\sigma \times u^{n-1}) \cdot v = \int (u^{n-1} \times v) \cdot \sigma = 2 \langle \sigma, l_3^* v \rangle. 
\]
Thus $l_3^* v = \frac{1}{2} u^{n-1} \times v$ on $H^1$, $l_3^*$ maps $H^1$ into itself by Lemma \ref{lemma:sobolev}, and 
\[l_3^* (K)_2 (W) \subset H^1 \subset H(\mathrm{curl}) = V^1.
\]
Thus we have 
$\Vert (\nabla \times) l_3^* (K)_2 \Vert_{W \rightarrow L^2} \lesssim \Vert (K)_2 \Vert_{W \rightarrow H^1}$
and by the boundedness of $\Vert K \Vert_{L^2 \times L^2 \rightarrow H^1 \times H^1} $ we get the boundedness of $\Vert d l_3^* (K)_2 \Vert_{W \rightarrow L^2}$.
Finally we have $\nu d d^* (K)_2 = I - d l_3^* (K)_2 - l_5^* (K)_2 - d^* (K)_3$ as distributions.
From the $L^2 \rightarrow L^2$ boundedness of the right-hand side we get both $\nu d^* (K)_2 (W) \subset H(\mathrm{curl}) = V^1$ and the boundedness of $\Vert d d^* (K)_2 \Vert_{W \rightarrow L^2}$.
The same argument applied to $\widebar{L}^{-1}$ shows that $d^* (\widebar{L}^{-1})_2(W) \subset V^1$.
Hence \eqref{eq:regularity35} is fulfilled.

\begin{remark}
Assuming $u^{n-1} \in H^2(\Omega)$ is very mild as any solution $u$ of \eqref{eq:defmixedprimal} must have $\nabla \times u \in H(\mathrm{curl},\Omega)$, $\nabla \cdot u = 0$ thus $\Delta u \in L^2(\Omega)$.
Hence by elliptic regularity for $\Omega$ smooth enough and if $u$ satisfies appropriate boundary conditions then $u \in H^2$.
\end{remark}

\subsection{Pointwise vanishing divergence and pressure-robustness} \label{Pressurerobustness}
The pointwise vanishing divergence is a simple fact that follows from the use of a discrete subcomplex.
From \eqref{eq:NS_Eulercont} we have $\forall q_h \in V^3_h$, $\langle \nabla \cdot u_h + \phi_h, q_h \rangle = 0$ 
with $\nabla \cdot u_h \perp \phi_h$ and $\nabla \cdot u_h \in V^3_h$ by construction.
Therefore, taking $q_h = \nabla \cdot u_h$ we have 
\[\langle \nabla \cdot u_h + \phi_h, \nabla \cdot u_h \rangle = \Vert \nabla \cdot u_h \Vert^2 = 0.
\]
This holds even for Condition \eqref{eq:bc2} or \eqref{eq:bc3} which does not give a complex structure.

A scheme is called pressure-robust (\cite{zbMATH07185373,Linke2020,Lederer2019}) if 
only the pressure (and not the velocity) changes when the external forces acting on the system are modified by a gradient.
This property is only valid if 
there are no harmonic $2$-forms.

Every vector field $f \in L^2(\Omega)$ can be written as $f = \nabla \times g + \nabla p$
for some fields $g$ and $p$.
In a bounded domain we only have uniqueness with correct boundary conditions on $g$ and $p$.
As long as these boundary conditions match the ones given in the complex, 
we have by viewing $f$ as a 2-form, $ \nabla \times g = P_\mathfrak{B} f$ 
and $\nabla p = P_{\mathfrak{B}^*} f$.
Let $\bar{f} \in L^2(\Omega)$ be another source term and
write $(\omega_h,u_h,p_h,\phi_h)$ (resp. $(\bar{\omega}_h,\bar{u}_h,\bar{p}_h,\bar{\phi}_h)$)
the solution of \eqref{eq:defmixeddiscrete} (resp. \eqref{eq:defmixeddiscrete}) for the external forces $(f,0)$ (resp. $(\bar{f},0)$).
\begin{theorem}\label{th:pressurerobust}
If $\exists s \in V^3$ such that $\bar{f} = f + \nabla s$, then
$\omega_h = \bar{\omega}_h$ and $u_h = \bar{u}_h$.
\end{theorem}
\begin{proof}
If there is $s \in V^3$ such that $f + \nabla s = \bar{f}$ then $P_\mathfrak{B} (f - \bar{f}) = 0$,
hence $\forall g \in V^1, \langle f - \bar{f}, \nabla \times g \rangle = 0$ and in particular
\[\forall g_h \in V^1_h \subset V^1,  \langle f - \bar{f}, \nabla \times g_h \rangle = 0\,.
\]
Therefore $P_{\mathfrak{B}_h} (f - \bar{f}) = 0$, and since we assumed that there were no harmonic $2$-forms,  $P_h (f - \bar{f}) = P_{\mathfrak{B}_h^*} (f - \bar{f})$.
Thus we can find $\xi_h \in V^3_h$ such that $\xi_h \perp \mathfrak{H}_h$ and $d^*_h \xi_h = - P_h (f - \bar{f})$.
Moreover $(0,0,\xi_h,0)$ verifies:
$\forall (\tau_h,v_h,q_h,\chi_h) \in V^1_h \times V^2_h \times V^3_h \times \mathfrak{H}_h$,
\begin{alignat*}{3}
\langle 0,\tau_h \rangle - \langle 0,\nabla \times \tau_h \rangle =&\ 0,& \quad
\langle \nu d 0 + l_3 0 + l_5 0, v_h \rangle - \langle \xi_h, \nabla \cdot v_h \rangle)
=&\ \langle f - \bar{f},v_h \rangle, \\
\langle \nabla \cdot 0,q_h \rangle + \langle 0,q_h \rangle  =&\ 0,& \quad
\langle \xi_h,\chi_h \rangle =&\ 0.
\end{alignat*}
By linearity and uniqueness of the solution we have 
$(\omega_h,u_h,p_h,\phi_h) = (\bar{\omega}_h,\bar{u}_h,\bar{p}_h,\bar{\phi}_h) + (0,0,\xi_h,0)$
and  $\omega_h = \bar{\omega}_h, u_h = \bar{u}_h$.
\end{proof}
\begin{remark}
Consider a time iterating scheme. 
At any time step $n$, the linear maps $l_3^n$ and $l_5^n$ involved in the determination of $(\omega^n_h,u^n_h,p^n_h,\phi^n_h)$ 
are functions of $u^{n-1}_h$ and $\omega^{n-1}_h$ (more precisely $l_3^n = l_3(u^{n-1}_h)$ and $l_5^n = l_5(\omega^{n-1}_h)$).
However, both $\omega_h$ and $u_h$ remain unchanged when the external force is modified by a gradient.
As long as $u^0_h = \bar{u}^0_h$ and $\omega^0_h = \bar{\omega}^0_h$,
an immediate recursion shows that $\forall n$, $u^n_h = \bar{u}^n_h$ and $\omega^n_h = \bar{\omega}^n_h$.
\end{remark}

The pressure-robustness allows us to remove the pressure from the error estimate.
\begin{theorem}[Pressure-robust estimate] \label{th:pressurerobustestimate}
Let $(\omega,u,p,\phi) := (u_1,u_2,u_3,u_p)$ be the solution of the continuous problem \eqref{eq:defmixedcontinuous}
and $(\omega_h,u_h,p_h,\phi_h) := (u_{1h},u_{2h},u_{3h},u_{ph})$ be the solution of the discrete problem \eqref{eq:defmixeddiscrete}.
Then it holds:
\begin{equation*}
\begin{aligned}
\Vert (u_1 - u_{1h},u_2 - u_{2h}) \Vert_{V^1 \times V^2} \lesssim&\, 
\inf_{v_1 \in V^1_h} \Vert u_1 - v_1 \Vert_{V^1} + \inf_{v_2 \in V^2_h} \Vert u_2 - v_2 \Vert_{V^2}, \\
\Vert \omega - \omega_h \Vert +  \Vert \nabla \times (\omega - \omega_h) \Vert +\Vert u - u_h \Vert \lesssim&\,
\inf_{\chi \in V^1_h} \Vert \omega - \chi \Vert_{V^1} + \inf_{v \in V^2_h} \Vert u - v \Vert_{V^2}.
\end{aligned}
\end{equation*}
\end{theorem}
\begin{proof}
Consider an alternative problem with the same $l_3$ and $l_5$ as before but with the source term replaced by $\tilde{f_2} := f_2 - \nabla p$.
Let $(\tilde{\omega},\tilde{u},\tilde{p},\tilde{\phi})$ and $(\tilde{\omega}_h,\tilde{u}_h,\tilde{p}_h,\tilde{\phi}_h)$ be respectively the continuous and discrete solution to this alternative problem.
By construction, we have $\tilde{p} = 0$, $\tilde{\phi} = 0$.
Hence, the estimate of Corollary \ref{corollary:pertubedestimate} gives:
\begin{equation*}
\Vert \tilde{\omega} - \tilde{\omega}_h \Vert +  \Vert \nabla \times (\tilde{\omega} - \tilde{\omega_h}) \Vert +\Vert \tilde{u} - \tilde{u_h} \Vert \lesssim
\inf_{\chi \in V^1_h} \Vert \tilde{\omega} - \chi \Vert_{V^1} + \inf_{v \in V^2_h} \Vert \tilde{u} - v \Vert_{V^2}.
\end{equation*}
We conclude since by the pressure-robustness we must have:
\begin{equation*}
\tilde{\omega} = \omega,\ \tilde{u} = u,\ 
\tilde{\omega}_h = \omega_h,\ \tilde{u}_h = u_h\ .
\end{equation*}
\end{proof}

\section{Numerical simulations} \label{Numericalsimulations}
We validate our scheme with three numerical simulations.
The first simulation aims to verify the pressure-robustness property.
The second is based on an exact and fully 3D solution of the Navier-Stokes equation constructed by Ethier \cite{Ethier1994}.
We use it to check the convergence rate in space, first on a steady problem then on an unsteady problem.
The last simulation focuses on Taylor-Couette flow: 
we seek the critical speed 
at which Taylor vortices appear.
It is based upon \cite{aitmoussa,Gebhardt1993,Hoffmann2005}.
In any case, we took a unit kinematic viscosity and polynomials of degree $2$.
The divergence is always checked to be pointwise zero up to machine precision.
Our codes are written with the FEniCS computing platform, version 2019.1.0 
(See \url{fenicsproject.org} and \cite{FEniCS})
and are available at \url{https://github.com/mlhanot/Navier-Stokes-feec}. 

\subsection{Pressure robustness}
We wish to verify the pressure robustness stated in Theorem \ref{th:pressurerobust},
i\@.e\@.,
that if the external forces acting on two flows differ only by a gradient,
then only the pressure differs between the flows.
We consider the Stokes no-flow problem in a glass (see \cite{SMAI-JCM_2019__5__89_0}).
The setup is rather simple. The mesh is a cylinder along the $z$ axis of height \num{2.}, base radius \num{1.} and top radius \num{1.5}, and the force $f$ derives from a potential:
\[f := \frac{\nabla \Phi} {\int_{\Omega} \Phi},\quad \Phi := z^\gamma
\]
for $\gamma = 1,2,4,7$.
We start from a fluid at rest and enforce a no slip condition on the whole boundary.
In every case we found a velocity equal to zero at the order of the machine precision.
This is not trivial as the same test conducted with Taylor-Hood elements ($P_2/P_1$) gave a significant nonzero velocity.
Moreover the error would scale with $\nu^{-1}$ (as shown in \cite{SMAI-JCM_2019__5__89_0}).

\subsection{Convergence rate to an exact solution}
We have conducted a convergence analysis with an exact solution.
The expression for the solution is given by Ethier \cite{Ethier1994} and depends on two real parameters $a$ and $d$. It is given by:
\[ u := \begin{bmatrix} 
-a (\exp(a x)\sin(a y + d z) + \exp(a z)\cos(a x + d y))\exp(-d^2 t)\\ 
-a (\exp(a y)\sin(a z + d x) + \exp(a x)\cos(a y + d z))\exp(-d^2 t)\\
-a (\exp(a z)\sin(a x + d y) + \exp(a y)\cos(a z + d x))\exp(-d^2 t)
\end{bmatrix}. 
\]
We have performed two sets of experiments: the first with $a = 2$ and $d = 0$ and the 
second with $a = 2$ and $d = 1$. 
The domain consists of a cylinder of height $2$ and radius $1$.
In the latter case the computation was done for $t$ between $0$ and $1$ 
with a time step of \num{1.e-3}.
We set the velocity to be equal to the velocity of the exact solution on the boundary (without enforcing any boundary condition on the vorticity).
For the stationary case ($d = 0$) we start from a fluid at rest, otherwise we start from the exact solution at $t = 0$.
We found a rate of convergence in space for the velocity of order $2.0$ in both cases, which is in agreement with the theory. 
Figure \ref{fig:Convrate} shows the convergence of the velocity in the relative $H(\text{div})$-norm $\frac{\Vert u_h - u \Vert_{H(\text{div})}}{\Vert u \Vert_{H(\text{div})}}$
with a log-log scale.
\begin{remark}
We took a time step small enough to neglect the error introduced by the time discretization.
\end{remark}
\begin{figure}
\centering
\begin{tikzpicture}[scale=0.65]
\begin{loglogaxis}[grid, xlabel=h, ylabel=Error, legend style={at={(0.0,1.0)},anchor=north west},
xtick={0.1,0.2,0.3,0.4,0.5,0.6,0.7,0.8,0.9},
xticklabels={0.1,0.2,0.3,0.4,0.5,0.6,0.7,0.8,0.9},
ytick={0.001,0.002,0.003,0.004,0.005,0.006,0.007,0.008,0.009,0.01},
yticklabels={\pgfmathprintnumber[sci]{0.001},\pgfmathprintnumber[sci]{0.002},,,,,,,,\pgfmathprintnumber[sci]{0.01}},
yticklabel style={/pgf/number format/.cd,sci}
]
\addplot plot coordinates {
(0.3701523923234168, 0.010096232860397174)
(0.2818874705382209, 0.0059842337475502399)
(0.22402549828914003, 0.0037725667958979302)
(0.1873969335628237, 0.0026026329284828472)
(0.16111550980350983, 0.0019131416970179491)
(0.14123099900924302, 0.0014777998363669485)
};
\addplot plot coordinates {
(0.37015239232341673, 0.0116568662938769)
(0.2790259042326483, 0.0065680806462421158)
(0.22581158624630765, 0.0042434222824466843)
(0.1875652678647778, 0.0029236973946968718)
(0.16089234492850563, 0.0021669763798863981)
(0.1407680298887053, 0.0016708889368816321)
};
\legend{$a = 2$, $d=0$\\$a=2$, $d=1$\\}
\logLogSlopeTriangle{0.90}{0.4}{0.1}{2}{black};
\end{loglogaxis}
\end{tikzpicture}
\vspace*{-1ex}
\caption{Convergence rate of the velocity with relative $H(\text{div})$ error on a log log scale.}
\label{fig:Convrate}
\end{figure}
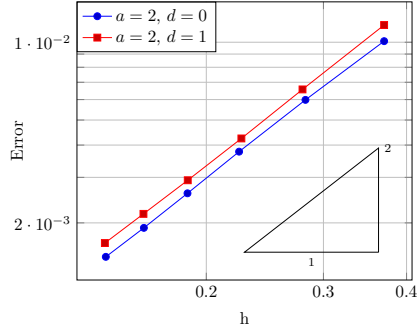
\subsection{Taylor-Couette flow} \label{TaylorCouetteflow}

\begin{figure}
\begin{minipage}[b]{0.65\textwidth}
\centering
\includegraphics[width=1\textwidth]{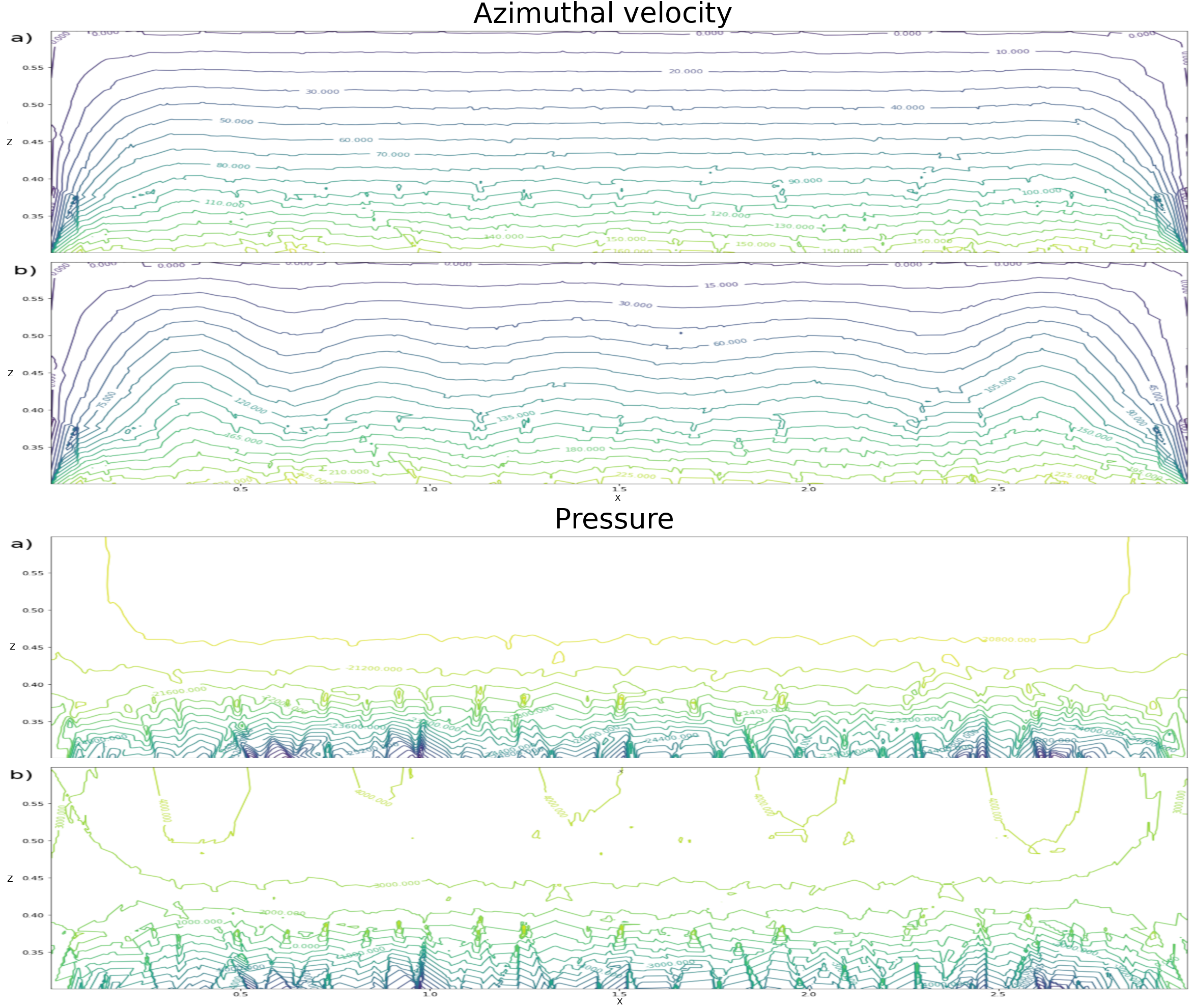}
\caption{Isolines for azimuthal velocity and pressure at $\eta = 0.5$, $Re_o = 0$, $Re_i = 50$ (a) and $Re_i = 72$ (b). }
\label{fig:5072}
\end{minipage}
\begin{minipage}[b]{0.3\textwidth}
\centering
\includegraphics[width=2\textwidth, angle=90]{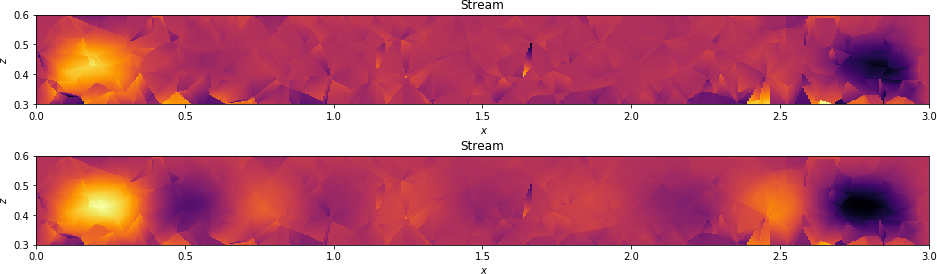}
\caption{Corresponding stream function.
}
\label{fig:Stream}
\end{minipage}
\end{figure}


\begin{figure}
\begin{minipage}[b]{0.48\textwidth}
\centering
\includegraphics[width=1.0\linewidth]{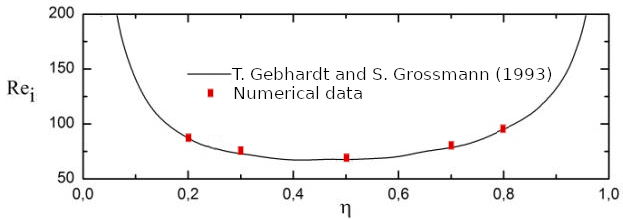}
\vspace*{-1ex}
\caption{Comparison of the critical value of $Re_i$ at $Re_o = 0$ for various $\eta$.}
\label{fig:3}
\end{minipage}
\begin{minipage}[b]{0.48\textwidth}
\centering
\includegraphics[width=1.0\linewidth]{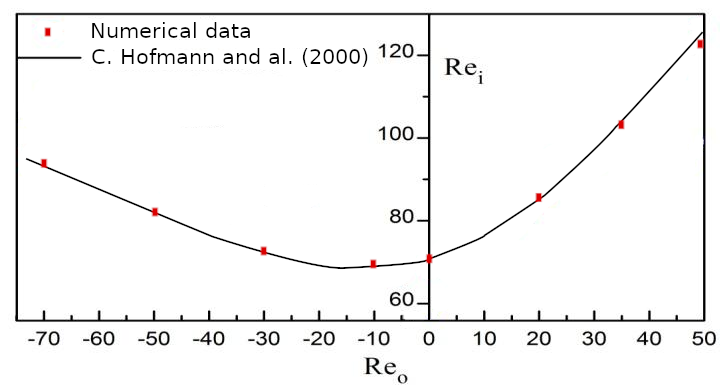}
\vspace*{-1ex}
\caption{Comparison of the critical value of $Re_i$ at $\eta = 0.5$ for various $Re_o$.}
\label{fig:5}
\end{minipage}
\end{figure}
This test focuses on Taylor-Couette flow.
We follow the work of Gebhardt et al.\ \cite{aitmoussa,Gebhardt1993,Hoffmann2005}.
The geometry consists in two concentric cylinders of constant radius $R_i$ for the inner and $R_o$ for the outer,
rotating at angular velocities $\Omega_i$ and $\Omega_o$ respectively, and both of height $a$.
The system is closed by two fixed lids at the bottom and top ends. We characterize the system by two geometric parameters: 
$\eta := R_i/R_o$ and $\Lambda := a/d$ with the gap $d := R_o - R_i$.
We also need to define two quantities: the inner Reynolds number $Re_i := \Omega_i R_i d /\nu$ and the outer Reynolds number $Re_o := \Omega_o R_o d/ \nu$ 
where $\nu$ is the kinematic viscosity.
For an infinite height $a$, it is a well known fact that,
at low speed the flow is steady and fully azimuthal,
and that vortices start to form at a critical speed.
Since $a$ is finite we expect to see vortices near the lids for speeds way under the critical value (they are however fundamentally different from the Taylor vortices, see \cite{Hoffmann2005}).

We compare the results obtained from our code with the reference \cite{Gebhardt1993,Hoffmann2005}. 
The simulations are done starting from a fluid at rest with a no slip boundary condition,
and studied at $t = 0.1$. 
Typically each simulation is done with a constant time step $\delta t = 0.001$. 
The time step was decreased for the simulations with the highest velocities.
We check the value of $Re_i$ at which the transition occurs for various values of $Re_o$ and $\eta$.
We display some values taken on the half plane $y = 0, x > 0$ 
for two values of $Re_i$ at $\eta = 0.5$, $Re_o = 0$.
Figure \ref{fig:5072} shows the azimuthal velocity and the pressure,
and Figure \ref{fig:Stream} shows the stream function (the azimuthal component of the vector potential).
The transition toward a Taylor-Couette flow occurs suddenly.
However, an automatic detection would be difficult to implement,
mainly due to interference from the fixed lids.
Hence, we obtained an interval rather than a point for the critical value of $Re_i$ 
(for each value of $Re_o$ and $\eta$).
Our results are represented by the red rectangles in Figures \ref{fig:3} and \ref{fig:5}.
Their height is given by the length of the interval (and their width is arbitrary).
The reference data are represented by the black curves.
We find very good agreement with the reference, 
even though we used a much coarser mesh and a smaller aspect ratio $\Lambda$ of $10$ instead of $20$, for computational cost reasons.

\appendix
\section{Improved error estimates} \label{Improvederrorestimates}
This section is dedicated to the proofs of the improved results of Theorem \ref{th:unpertubedestimate}
which are used in the proof of Lemma \ref{lemma:estimatesforzh}.
We will make use of the following lemma (see \cite[Lemma~3.12]{Arnold_2010}):
\begin{lemma} \label{lemma:lemma12}
Let $v_h \in \mathfrak{Z}_h^{k \perp}$ and $v = P_{\mathfrak{B}^*} v_h$. Then 
\begin{equation*}
\Vert v - v_h \Vert \leq \Vert (I - \pi_h^k) v \Vert \leq \eta_0' \Vert d v_h \Vert ,
\end{equation*}
where $\eta_0'$ is the scalar given in Lemma \ref{lemma:PIh}.
\end{lemma}
\begin{remark} \label{rem:th37HodgeTheory}
As a consequence of the proof of \cite[Theorem~3.7]{Arnold_2010} we also see that 
we can take $\pi_h$ such that
for $u \in V^k$, 
$\Vert (I - P_{\mathfrak{B}_h}) d u \Vert = \Vert (I - \pi_h) d u \Vert \lesssim E(d u)$.
\end{remark}
We can now proceed to the various proofs.
\begin{theorem} \label{th:improvedestimates}
Let $(u_1,u_2,u_3,\phi_2,\phi_3) \in V^1 \times V^2 \times V^3 \times \mathfrak{H}^2 \times \mathfrak{H}^3$ be the solution of the continuous problem \eqref{eq:pbB0} 
and let $(u_{1h},u_{2h},u_{3h},\phi_{2h},\phi_{3h}) \in V^1_h \times V^2_h \times V^3_h \times \mathfrak{H}^2_h \times \mathfrak{H}^3_h$ be the solution of the discrete problem \eqref{eq:discretelinearsteady}.
Then the following estimates hold:
\begin{align}
\Vert u_1 - u_{1h} \Vert \lesssim\,& E(u_1) + \eta_0 E(d u_1), \label{eq:impest1}\\
\Vert du_1 - du_{1h} \Vert \lesssim\,& E(d u_1), \label{eq:impest2}\\
\Vert u_3 - u_{3h} \Vert \lesssim\,& E(u_3) + \eta_0' (E(d u_1) + E(\phi_2)), \label{eq:impest3}\\
\Vert P_{\mathfrak{B}_h} (u_2 - u_{2h}) \Vert \lesssim\,& \eta_0 E(u_1) + (\eta_0^2 + \delta_0) E(du_1), \label{eq:impest4}\\
\Vert P_{\mathfrak{B}^*_h} (u_2 - u_{2h}) \Vert \lesssim\,& E(u_2) + E(du_2), \label{eq:impest5}\\
\Vert du_2 - du_{2h} \Vert \lesssim\,& E(du_2). \label{eq:impest6}
\end{align}
\end{theorem}
For the sake of readability we divide the proof of Theorem \ref{th:improvedestimates} in the following lemmas.
\begin{lemma}
The estimate \eqref{eq:impest1} holds.
\end{lemma}
\begin{proof}
We use the operator $K_0$ defined in Section \ref{Discretewellposedness} and recall the decomposition \eqref{eq:decompK}.
We write $u_1 := u_1^1 + u_1^2$ with $u_1^1$ and $u_1^2$ defined as follows:
\begin{equation*}
u_1 = d^* (K_0)_2(d u_1,0) = d^* (K_0)_2((I - P_{\mathfrak{B}_h}) d u_1,0) + d^* (K_0)_2(P_{\mathfrak{B}_h} d u_1,0) := u_1^1 + u_1^2.
\end{equation*}
Let $v_{1h} = \pi_h u_1^2 - u_{1h}$, we have:
\begin{equation*}
\begin{aligned}
d v_{1h} =\;& \pi_h d u_1^2 - d u_{1h}\\
=\;& \pi_h d d^* (K_0)_2 (P_{\mathfrak{B}_h} d u_1, 0) - d u_{1h}\\
=\;& \pi_h P_\mathfrak{B} P_{\mathfrak{B}_h} d u_1 - d u_{1h}\\
=\;& \frac{1}{\nu} (P_{\mathfrak{B}_h} f_2 - P_{\mathfrak{B}_h} f_2) \\
=\;& 0.
\end{aligned}
\end{equation*}
Moreover injecting $v_{1h}$ into \eqref{eq:pbB0} and \eqref{eq:discretelinearsteady} gives:
\[
\langle u_1 - u_{1h}, v_{1h} \rangle = \langle u_2 - u_{2h}, d v_{1h} \rangle = 0.
\]
Hence, we have:
\begin{equation*}
\begin{aligned}
\langle u_1 - u_{1h}, u_1 - u_{1h} \rangle =\;& \langle u_1 - u_{1h}, \pi_h u_1^2 - u_{1h} \rangle + \langle u_1 - u_{1h}, u_1 - \pi_h u_1^2 \rangle\\
=\;& 0 +  \langle u_1 - u_{1h}, u_1 - \pi_h u_1^2 \rangle\\
\leq\;& \Vert u_1 - u_{1h} \Vert \Vert u_1 - \pi_h u_1^2 \Vert
\end{aligned}
\end{equation*}
and
\begin{equation*}
\Vert u_1 - \pi_h u_1^2 \Vert \leq \Vert (I - \pi_h) u_1 \Vert + \Vert \pi_h u_1^1 \Vert \lesssim E(u_1) + \Vert \pi_h u_1^1 \Vert.
\end{equation*}
It remains to bound $\Vert \pi_h u_1^1 \Vert$, and since 
\[ u_1^1 = d^* (K_0)_2 (d u_1^1, 0) = d^* (K_0)_2((I - P_{\mathfrak{B}_h}) d u_1, 0),
\]
Remark \ref{rem:th37HodgeTheory} gives:
\begin{equation*}
\begin{aligned}
\langle u_1^1, u_1^1 \rangle =\;& \langle (K_0)_2 (d u_1^1,0), d u_1^1 \rangle \\
=\;& \langle (K_0)_2 (d u_1^1,0),(I - P_{\mathfrak{B}_h}) d u_1 \rangle \\
=\;& \langle (I - P_{\mathfrak{B}_h}) (K_0)_2 (d u_1^1,0),(I - P_{\mathfrak{B}_h}) d u_1 \rangle \\
\leq\;& \Vert (I - \pi_h) (K_0)_2 (d u_1^1,0)\Vert \Vert(I - \pi_h) d u_1 \Vert \\
\lesssim\;& \eta_0 \Vert  u_1^1 \Vert E(d u_1).
\end{aligned}
\end{equation*}
\end{proof}
\begin{lemma} \label{lemma:impest2}
The estimate \eqref{eq:impest2} holds.
\end{lemma}
\begin{proof}
Using the Hodge decomposition of $f_2$ we have $d u_{1h} = P_{\mathfrak{B}_h} f_2 = P_{\mathfrak{B}_h} P_\mathfrak{B} f_2 = P_{\mathfrak{B}_h} d u_1$.
So by Remark \ref{rem:th37HodgeTheory} we have $\Vert d u_1 - d u_{1h} \Vert = \Vert (I - P_{\mathfrak{B}_h}) d u_1 \Vert \lesssim E(du_1) $.
\end{proof}
\begin{lemma}
The estimate \eqref{eq:impest3} holds.
\end{lemma}
\begin{proof}
For any $v_{2h} \in V^2_h$ we have 
\[
\nu \langle d (u_1 - u_{1h}), v_{2h} \rangle - \langle u_3 - u_{3h}, d v_{2h} \rangle + \langle \phi_2 - \phi_{2h}, v_{2h} \rangle = 0.
\]
Since $u_{3h} \perp \mathfrak{H}_h^3$, 
there is 
$v_{2h} \in \mathfrak{B}_{2h}^*$ such that $d v_{2h} = P_{\mathfrak{B}_h} u_3 - u_{3h}$.
Let $v_2 = P_{\mathfrak{B}^*} v_{2h}$.
Using the orthogonality of the Hodge decomposition we have $\langle d (u_1 - u_{1h}), v_2 \rangle = 0$ and $\langle \phi_2 - P_{\mathfrak{H}_h} \phi_2,v_2 \rangle = 0$ 
(since $\mathfrak{H}_h \subset \mathfrak{Z}_h \subset \mathfrak{Z} \perp \mathfrak{B}^* \ni v_2$). \\
Moreover $\langle \phi_2 - \phi_{2h}, v_{2h} \rangle = \langle \phi_2, v_{2h} \rangle = \langle \phi_2 - P_{\mathfrak{H}_h} \phi_2, v_{2h} \rangle$, so it holds:
\begin{equation*}
\begin{aligned}
\langle u_3 - u_{3h}, P_{\mathfrak{B}_h} u_3 - u_{3h} \rangle 
=\;& \langle d(u_1 - u_{1h}), v_{2h} \rangle + \langle \phi_2 - \phi_{2h}, v_{2h} \rangle \\
=\;& \langle d(u_1 - u_{1h}), v_{2h} - v_2 \rangle + \langle \phi_2 - P_{\mathfrak{H}_h} \phi_{2}, v_{2h} - v_2\rangle \\
\leq\;&  ( \Vert d(u_1 - u_{1h}) \Vert + \Vert \phi_2 - P_{\mathfrak{H}_h} \phi_2 \Vert ) \Vert v_2 - v_{2h} \Vert \\
\lesssim\;& (E( d u_1) + E(\phi_2)) \eta_0' \Vert P_{\mathfrak{B}_h} u_3 - u_{3h} \Vert.
\end{aligned}
\end{equation*}
We used Lemma \ref{lemma:lemma12} for the last inequality. 
Hence we have:
\begin{equation*}
\begin{aligned}
\Vert P_{\mathfrak{B}_h} u_3 - u_{3h} \Vert^2 
=\;& \langle P_{\mathfrak{B}_h} u_3 - u_3 , P_{\mathfrak{B}_h} u_3 - u_{3h} \rangle + \langle u_3 - u_{3h} , P_{\mathfrak{B}_h} u_3 - u_{3h} \rangle\\
\lesssim\;& (\Vert P_{\mathfrak{B}_h} u_3 - u_3 \Vert + \eta_0' (E(d u_1) + E( \phi_2)) ) \Vert P_{\mathfrak{B}_h} u_3 - u_{3h} \Vert\\
\lesssim\;& (E(u_3) + \eta_0' (E(d u_1) + E(\phi_2)) ) \Vert P_{\mathfrak{B}_h} u_3 - u_{3h} \Vert.
\end{aligned}
\end{equation*}
Now, since $\Vert P_h u_3 - P_{\mathfrak{B}_h} u_3 \Vert = \Vert P_{\mathfrak{H}_h} u_3 \Vert \lesssim \mu_0 E(u_3)$ and $\mu_0 \lesssim 1$ we have:
\begin{equation*}
\begin{aligned}
\Vert u_3 - u_{3h} \Vert \lesssim\;& \Vert u_3 - P_h u_3 \Vert + \Vert P_h u_3 - P_{\mathfrak{B}_h} u_3 \Vert + \Vert P_{\mathfrak{B}_h} u_3 - u_{3h} \Vert \\
\lesssim\;& E(u_3) + \eta_0' (E(d u_1) + E(\phi_2)) .
\end{aligned}
\end{equation*}
\end{proof}
\begin{lemma}
The estimate \eqref{eq:impest4} holds.
\end{lemma}
\begin{proof}
Let $e = P_{\mathfrak{B}_h}(u_2 - u_{2h})$, 
$w = (K_0)_2(e,0)$,
$\psi = d^* w$,
$w_h = (K_{0h})_2(e,0)$ and
$\psi_h = d_h^* w_h$.
First notice that 
\[ d \psi = dd^*(K_0)_2 (e,0) = P_\mathfrak{B} e = e = d \psi_h,
\]
so $d \psi \in V_h^2$ and $d \psi = \pi_h d \psi = d \pi_h \psi = d \psi_h$.
Moreover $\pi_h \psi - \psi_h \perp \psi - \psi_h$, indeed:
\begin{equation*}
\begin{aligned}
\langle \pi_h \psi - \psi_h, \psi \rangle &= \langle \pi_h \psi - \psi_h, d^* w \rangle = \langle d \pi_h \psi - d\psi_h, w \rangle = 0,\\
\langle \pi_h \psi - \psi_h, \psi_h \rangle &= \langle \pi_h \psi - \psi_h, d_h^* w_h \rangle = \langle d \pi_h \psi - d\psi_h, w_h \rangle = 0.
\end{aligned}
\end{equation*}
This allows us to derive the following bounds:
\begin{equation*}
\begin{aligned}
\langle \psi - \psi_h, \psi - \psi_h \rangle 
=&\; \langle \psi - \pi_h \psi, \psi - \psi_h \rangle + \langle \pi_h \psi - \psi_h, \psi - \psi_h\rangle\\
=&\; \langle \psi - \pi_h \psi, \psi - \psi_h \rangle\\
\leq&\; \Vert (I - \pi_h) \psi \Vert \Vert \psi - \psi_h \Vert,
\end{aligned}
\end{equation*}
\begin{equation*} 
\Vert \psi - \psi_h \Vert \leq \Vert (I - \pi_h) \psi \Vert = \Vert (I - \pi_h) d^*(K_0)_2(e,0) \Vert \leq \eta_0 \Vert e \Vert .
\end{equation*}
Finally, recalling that $d u_{1h} = P_{\mathfrak{B}_h} d u_1$ (Lemma \ref{lemma:impest2}), we have:
\begin{equation*}
\begin{aligned}
\Vert e \Vert^2 =\;& \langle d \psi_h, e \rangle = \langle d \psi_h, u_2 - u_{2h} \rangle = \langle u_1 - u_{1h}, \psi_h \rangle \\
=\;& \langle u_1 - u_{1h}, \psi_h - \psi \rangle + \langle u_1 - u_{1h}, \psi \rangle\\
=\;& \langle u_1 - u_{1h}, \psi_h - \psi \rangle + \langle du_1 - du_{1h}, w \rangle\\
=\;& \langle u_1 - u_{1h}, \psi_h - \psi \rangle + \langle (I - P_{\mathfrak{B}_h}) du_1, (I - P_{\mathfrak{B}_h}) w \rangle\\
\leq\;& \Vert u_1 - u_{1h} \Vert \Vert \psi_h - \psi \Vert + \Vert (I - P_{\mathfrak{B}_h}) du_1 \Vert \Vert (I - P_{\mathfrak{B}_h}) w \Vert\\
\lesssim\;& \Vert u_1 - u_{1h} \Vert \Vert \psi_h - \psi \Vert + \Vert (I - \pi_h) du_1 \Vert \Vert (I - \pi_h) w \Vert\\
\lesssim\;& \Vert u_1 - u_{1h} \Vert \eta_0 \Vert e \Vert + E(d u_1) \delta_0 \Vert e \Vert.\\
\end{aligned}
\end{equation*}
We conclude with the estimate \eqref{eq:impest1}.
\end{proof}
\begin{lemma}
The estimate \eqref{eq:impest6} holds.
\end{lemma}
\begin{proof}
We know that $du_2 = P_{\mathfrak{B}} f_3$ and $d u_{2h} = P_{\mathfrak{B}_h} f_3 = P_{\mathfrak{B}_h} d u_2$,
hence 
\[\Vert d u_2 - d u_{2h} \Vert = \Vert (I - P_{\mathfrak{B}_h}) d u_2 \Vert \lesssim E(d u_2). 
\]
\end{proof}
\begin{lemma}
The estimate \eqref{eq:impest5} holds.
\end{lemma}
\begin{proof}
Let $v_{2h} = P_{\mathfrak{B}_h^*} (\pi_h u_2 - u_{2h} )$.
The result follows from triangular inequalities, Poincaré inequalities and \eqref{eq:impest6} since:
\begin{equation*}
\begin{aligned}
\Vert v_{2h} \Vert \leq\;& c_p \Vert d (\pi_h u_2 - u_{2h}) \Vert \\
\leq\;& c_p \Vert \pi_h d u_2 - d u_2 + d u_2 - d u_{2h} \Vert\\
\leq\;& c_p \Vert (I - \pi_h) d u_2 \Vert + c_p \Vert d u_2 - d u_{2h} \Vert \\
\lesssim\;& E(d u_2).
\end{aligned}
\end{equation*}
Hence: 
\begin{equation*}
\begin{aligned}
\Vert P_{\mathfrak{B}_h^*} ( u_2 - u_{2h}) \Vert 
\leq\;& \Vert P_{\mathfrak{B}_h^*} (u_2 - \pi_h u_2) \Vert + \Vert v_{2h} \Vert \\
\lesssim\;& E(u_2) + E(d u_2).
\end{aligned}
\end{equation*}
\mbox{}
\end{proof}
\begin{corollary}
Keeping the notation of Theorem \ref{th:improvedestimates}, it holds:
\begin{equation*}
\Vert u_2 - u_{2h} \Vert \lesssim E(u_2) + E (d u_2) + \eta_0 E( u_1) + (\eta_0^2 + \delta_0) E(d u_1) .
\end{equation*}
\end{corollary}
\begin{proof}
This is a direct consequence of the estimates \eqref{eq:impest4}, \eqref{eq:impest5} and Lemma \ref{lemma:erroronharmonics}.
We simply write:
\begin{equation*}
\begin{aligned}
\Vert u_2 - u_{2h} \Vert \leq\;& \Vert u_2 - P_h u_2 \Vert + \Vert P_{\mathfrak{B}_h}(u_2 - u_{2h}) \Vert + \Vert P_{\mathfrak{H}_h} u_2 \Vert + \Vert P_{\mathfrak{B}_h^*}(u_2 - u_{2h})\Vert \\
\lesssim\;& E(u_2) + \eta_0 E(u_1) +(\eta_0^2 + \delta_0) E(d u_1) + \mu_0 E(u_2) + E(u_2) + E(d u_2).
\end{aligned}
\end{equation*}
\mbox{}
\end{proof}

\section{Time step independent estimates}
When we apply the scheme to solve iteratively a problem
some details are not accounted for in the estimate of Corollary \ref{corollary:pertubedestimate}.
Indeed $\Vert l_5\Vert$ and $\Vert f_2 \Vert$ depend on $\delta t$, 
and the linear maps $l_3$ and $l_5$ are not the same in the continuous and in the discrete problem.

Here we provide an estimate independent of $\delta t$,
and with two different sets of linear maps for a single step in time.
More precisely let $(u_1,u_2,u_3,u_p)$ be the solution of the continuous problem \eqref{eq:defmixedcontinuous} for
$f_2 := \frac{1}{\delta t} f^c + g$, $f_3 := 0$, $l_3 := l_3^c$, $l_5 := \frac{1}{\delta t} I + l_5^c$.
And let $(u_{1h},u_{2h},u_{3h},u_{ph})$ be the solution of the discrete problem \eqref{eq:defmixeddiscrete} for 
$f_2 := \frac{1}{\delta t} f^d + g$, $f_3 := 0$, $l_3 := l_3^d$, $l_5 := \frac{1}{\delta t} I + l_5^d$.
For $i \in \lbrace 1,2,3 \rbrace$ we define $e_i := u_i - u_{ih}$ 
as well as $e_f := f^c - f^d$, $e_{l_3} := l_3^c - l_3^d$ and $e_{l_5} := l_5^c - l_5^d$.
\begin{theorem} \label{th:timeindeest}
There exists $\delta t_0 > 0$ depending only on $\Vert l_3^d\Vert$, $\Vert l_5^d\Vert$ and $\nu$ 
such that $\forall \delta t := \alpha\, \delta t_0$ with $0 < \alpha < 1$, 
\begin{equation*}
\begin{aligned}
\Vert P_{\mathfrak{Z}_h} e_2 \Vert \leq&\; 
 \frac{\delta t}{1 - \alpha} \bigg( \Vert l_5^d \Vert \Vert (I - P_{\mathfrak{Z}_h}) u_2 \Vert + \Vert l_3^d \Vert \Vert (I - P_h) u_1 \Vert  \\
&\quad\  + \Vert e_{l_3} \Vert \Vert u_1 \Vert + \Vert e_{l_5} \Vert \Vert u_2 \Vert \bigg) + \frac{1}{1 - \alpha} \Vert e_f \Vert ,\\
\Vert e_2 \Vert \lesssim&\; \Vert P_{\mathfrak{Z}_h} e_2 \Vert + E(u_2) .
\end{aligned}
\end{equation*}
\end{theorem}
\begin{proof}
Subtracting \eqref{eq:defmixeddiscrete} from \eqref{eq:defmixedcontinuous} gives: 
\begin{equation} \label{eq:thestimateindepdt1}
\begin{aligned}
\langle e_1, v_{1h} \rangle =& \langle e_2, d v_{1h} \rangle,\\
\langle \nu d e_1 + e_{l_3} u_1 + l_3^d e_1 + e_{l_5} u_2 + l_5^d e_2 + \frac{1}{\delta t} e_2, v_{2h} \rangle =& \langle e_3, d v_{2h} \rangle + \langle \frac{1}{\delta t} e_f, v_{2h} \rangle.
\end{aligned}
\end{equation}
We take respectively $v_{1h} = P_h e_1$ and $v_{1h} = d_h^* P_{\mathfrak{Z}_h} e_2$ in \eqref{eq:thestimateindepdt1} to get:
\begin{equation} \label{eq:thestimateindepdt2}
\begin{aligned}
\Vert P_h e_1 \Vert^2 = \langle e_1, P_h e_1 \rangle =& \langle e_2, d P_h e_1 \rangle = \langle P_{\mathfrak{Z}_h} e_2, d P_h e_1 \rangle,\\
\langle d P_h e_1, P_{\mathfrak{Z}_h} e_2 \rangle 
= \langle e_1, d_h^* P_{\mathfrak{Z}_h} e_2 \rangle =& \langle e_2, d d_h^* P_{\mathfrak{Z}_h} e_2 \rangle  = \Vert d_h^* P_{\mathfrak{Z}_h} e_2 \Vert^2.
\end{aligned}
\end{equation}
Taking $v_{2h} = P_{\mathfrak{Z}_h} e_2$ in \eqref{eq:thestimateindepdt1}
and making use of \eqref{eq:thestimateindepdt2}, $d P_{\mathfrak{Z}_h} e_2 = 0$ and \\ 
$\vert \langle l_3^d P_h e_1, P_{\mathfrak{Z}_h} e_2 \rangle \vert \leq \frac{\nu}{2} \Vert P_h e_1 \Vert^2 + \frac{\Vert l_3^d \Vert^2}{2 \nu} \Vert P_{\mathfrak{Z}_h} e_2\Vert^2$, 
we get:
\begin{equation} \label{eq:thestimateindepdt3}
\begin{gathered}
\frac{\nu}{2} \Vert P_h e_1 \Vert^2 + \left ( \frac{1}{\delta t} - \frac{\Vert l_3^d \Vert^2}{2 \nu} - \Vert l_5^d\Vert \right ) \Vert P_{\mathfrak{Z}_h} e_2 \Vert^2 
\leq \\
\left ( \Vert l_5^d \Vert \Vert (I - P_{\mathfrak{Z}_h}) e_2 \Vert + \Vert l_3^d \Vert \Vert ( I - P_h) e_1 \Vert
+ \Vert e_{l_3} \Vert \Vert u_1 \Vert + \Vert e_{l_5} \Vert \Vert u_2 \Vert + \frac{1}{\delta t} \Vert e_f \Vert 
\right ) \Vert P_{\mathfrak{Z}_h} e_2 \Vert .
\end{gathered}
\end{equation}
Let 
\[\delta t_0 := \frac{2 \nu}{\Vert l_3^d \Vert^2 + 2 \nu \Vert l_5^d\Vert}. \]
We conclude since 
it holds $(I - P_h) e_1 = (I - P_h) u_1$,
and since $u_{2h} \in \mathfrak{Z}_h$, so $(I - P_{\mathfrak{Z}_h}) e_2 = (I - P_{\mathfrak{Z}_h}) u_2$. 
Moreover $u_2 \in \mathfrak{Z}$ and $u_{2h} \in \mathfrak{Z}_h \subset \mathfrak{Z}$, so
\[\Vert e_2 \Vert = \Vert P_{\mathfrak{Z}} e_2 \Vert \leq \Vert (P_{\mathfrak{Z}} - P_{\mathfrak{Z}_h}) e_2 \Vert + \Vert P_{\mathfrak{Z}_h} e_2 \Vert ,\]
\[\Vert (P_{\mathfrak{Z} } - P_{\mathfrak{Z}_h}) e_2 \Vert = \Vert (P_{\mathfrak{Z}} - P_{\mathfrak{Z}_h}) u_2 \Vert \lesssim \Vert (I - \pi_h) P_{\mathfrak{Z}} u_2 \Vert 
\lesssim E(P_{\mathfrak{Z}} u_2) = E(u_2) .\]
\end{proof}

\section*{Acknowledgements} I thank Pascal Azerad for his
comments and helpful advice.
I also thank Kaibo Hu for his comments and suggestions
and the anonymous referees for many helpful remarks.

\printbibliography

\end{document}